\documentclass[11pt]{article}
\usepackage{amssymb,amsmath,amsthm}
\usepackage{enumerate}
\usepackage{pgf,tikz}
\usetikzlibrary{arrows}

\newtheorem{theorem}{Theorem}[section]
\newtheorem{corollary}[theorem]{Corollary}
\newtheorem{lemma}[theorem]{Lemma}
\newtheorem{proposition}[theorem]{Proposition}

\newtheorem{definition}[theorem]{Definition}

\newtheorem{remark}[theorem]{Remark}
\newtheorem{fact}[theorem]{Fact}
\numberwithin{equation}{section}

\DeclareMathOperator{\co}{co}
\DeclareMathOperator{\Des}{Des}

\newcommand{\SSSS}{\mathfrak S}
\newcommand{\ab}{\av\bv}
\newcommand{\av}{{\bf a}}
\newcommand{\bv}{{\bf b}}
\newcommand{\cd}{\cv\dv}
\newcommand{\cv}{{\bf c}}
\newcommand{\dv}{{\bf d}}

\newcommand{\zab}{\Zzz\langle\av,\bv\rangle}

\newcommand{\Rrr}{\mathbb{R}}
\newcommand{\Qqq}{\mathbb{Q}}
\newcommand{\Zzz}{\mathbb{Z}}

\begin{document}

\author{Richard Ehrenborg and N.\ Bradley Fox}
\title{The Descent Set Polynomial Revisited}
\date{}

\maketitle

\begin{abstract}
We continue to explore cyclotomic factors
in the descent set polynomial $Q_{n}(t)$,
which was introduced by
Chebikin, Ehrenborg, Pylyavskyy and Readdy.
We obtain large classes of factors of the form
$\Phi_{2s}$ or $\Phi_{4s}$ where $s$ is an odd integer,
with many of these being of the form
$\Phi_{2p}$ where $p$ is a prime.
We also show that if $\Phi_{2}$ is a factor
of $Q_{2n}(t)$ then it is a double factor.
Finally, we give conditions for an odd prime power
$q = p^{r}$ for which 
$\Phi_{2p}$ is a double factor
of~$Q_{2q}(t)$ and of $Q_{q+1}(t)$.
\end{abstract}

\section{Introduction}

For a permutation $\pi$ in the symmetric group~$\SSSS_{n}$,
define the descent set
of $\pi$ to be the subset of $[n-1] = \{1,2, \ldots, n-1\}$
given by
$\Des(\pi) = \{i \in [n-1] \: : \: \pi_{i} > \pi_{i+1}\}$.
The descent set statistics~$\beta_{n}(S)$ are defined for
subsets $S$ of $[n-1]$ by
$$  \beta_{n}(S)
   =
      \left|
        \{ \pi \in \SSSS_{n} \: : \: \Des(\pi) = S \}
      \right|   .   $$
Chebikin, Ehrenborg, Pylyavskyy and Readdy~\cite{C_E_P_R}
defined the $n$th descent set polynomial to be
$$    Q_{n}(t)
    =
        \sum_{S \subseteq [n-1]} t^{\beta_{n}(S)} .  $$
They observed that this polynomial has many factors
that are cyclotomic polynomials.
The most common of these cyclotomic polynomials
is $\Phi_{2} = t+1.$
It is direct that having $\Phi_{2}$ as a factor implies that
the number of subsets of $[n-1]$ having an even descent set
statistic is the same as
the number of subsets having an odd descent set
statistic.
Consider the proportion
of odd entries among the descent set statistics in the 
symmetric group $\mathfrak{S}_n$, that is,
$$
  \rho(n)
    =
  \frac{|\{S\subseteq [n-1]\text{ : } \beta_n(S)\equiv 1 \bmod 2\}|}{2^{n-1}} .
$$ 
Chebikin et al.\ showed that this proportion depends
on the number of $1$'s in the binary expansion
of~$n$. We quote their paper with the following table.
Only the values $2^k-1$ are included in the table since
$\rho(2^k-1)$ is the same as $\rho(n)$ if $n$ has
$k$ 1's in its binary expansion.
\begin{table}[h]
$$
\begin{array}{c | c c c c c}
n & 1 & 3 & 7 & 15 & 31 \\ \hline
\rho(n) & 1 & {1}/{2} & {1}/{2} & {29}/{2^{6}} & {3991}/{2^{13}} 
\end{array}
$$
\caption{The proportion $\rho(n)$.}
\label{table_rho}
\end{table}
Hence when $n$ has two or three $1$'s in its binary
expansion we obtain $\Phi_{2}$ as a factor
in the descent set polynomial $Q_{n}(t)$.
Note that the proportion is not known for six or more $1$'s
in the binary expansion.

Chebikin et al.\ gave more results for cyclotomic factors
in the descent set polynomial:
\begin{enumerate}[(i)]
\item
When $n = 2^{j} \geq 4$ then $\Phi_{4}$ divides $Q_{n}(t)$.

\item
When $q = p^{r}$ is an odd prime power with
two or three $1$'s in its binary expansion 
and $q \neq 3$ or~$7$,
then
$\Phi_{2p}$ divides $Q_{q}(t)$.

\item
When $q = p^{r}$ is an odd prime power with
two or three $1$'s in its binary expansion, 
then
$\Phi_{2p}$ divides $Q_{2q}(t)$.
\end{enumerate}
They also found cases when
there were double factors in the descent set polynomial:
\begin{enumerate}
\item[(iv)]
If the binary expansion of $n$ has two $1$'s in its binary
expansion and $n > 3$,
then $\Phi_{2}$
is a double factor of $Q_{n}(t)$.

\item[(v)]
If $n = 2^{j} \geq 4$ then $\Phi_{4}$
is a double factor of $Q_{n}(t)$.

\item[(vi)]
When $q = p^{r}$ is an odd prime power
and $q$ has two $1$'s in its binary expansion, 
then $\Phi_{2p}$ 
is a double factor of $Q_{n}(t)$.
\end{enumerate}

We continue their work in explaining cyclotomic factors
in these polynomials.  
In Section~\ref{section_preliminaries}
we review some preliminary notions
and tools that will help
in developing our results.
We introduce a simplicial complex
in Section~\ref{section_complex}
that determines the parity of the 
descent statistics.
Namely, the reduced Euler characteristic
of an induced subcomplex gives the
descent statistics modulo~$2$.
In Section~\ref{section_one_digit} we 
prove for $s$ an odd integer when $\Phi_{4s}$ 
is a factor of $Q_{n}(t)$ with $n$ being a power
of $2$.
In Section~\ref{section_two_digits}
we show
for $s$ an odd integer when $\Phi_{2s}$ is
a factor of $Q_{n}(t)$ when $n$ has two
non-zero digits in its binary expansion.
We prove a multitude of cases in this section
when we set $s$ to be a prime number $p$.
Similarly, when $n$ has three digits in
its binary expansion, we develop cases
when $\Phi_{2s}$, and likewise $\Phi_{2p}$, is
a factor of $Q_{n}(t)$
in Section~\ref{section_three_digits}.

We also continue the work on double factors
in the descent set polynomial $Q_{n}(t)$
in Sections~\ref{section_double_factor_Phi_2}
through~\ref{section_q_plus_1}.
In fact, the two results (iv) and (vi) both need 
the condition that
the number of $1$'s in the binary expansion of~$n$
is exactly two.
Furthermore, the result (vi) applies only (so far) to the five Fermat
primes and the prime power~$3^{2}$,
whereas our results apply
when there are two or three $1$'s
in the binary expansion.
First in Theorem~\ref{theorem_double_factor}
we show that if $\Phi_{2}$ is a factor of
$Q_{2n}(t)$ then it is a double factor.
Next in
Theorems~\ref{theorem_double_factor_two}
and~\ref{theorem_double_factor_three}
we find the double factor $\Phi_{2p}$ in
$Q_{2q}(t)$ and $Q_{q+1}(t)$ where $q = p^{r}$
is an odd prime power.
The corresponding proofs in~\cite{C_E_P_R}
depend on substituting 
values for the variables in the $\ab$-index
of the Boolean algebra,
whereas our proofs rely on evaluating
a more general linear function;
see Proposition~\ref{proposition_2n_n}.
The underlying reason for these results is
that the descent set statistic
is straightforward to compute modulo the prime $p$;
see Lemma~\ref{lemma_q} and
equation~\eqref{equation_q_plus_1}.

A summary of cyclotomic factors of $Q_{n}(t)$ that Chebikin et al.\  found, as well as 
which ones were explained by their and our results,
can be found in Table~\ref{table_P}.
We end with open questions in the concluding remarks.

\section{Preliminaries}        
\label{section_preliminaries}

Let $[i,j]$ denote the interval $\{i, i+1, \ldots, j\}$.
Furthermore, let $\triangle$ denote the symmetric
difference of two sets, that is,
$S \triangle T = S \cup T - S \cap T$.
Finally, let $S - k$ denote the shifting of the set by $k$,
that is, $S - k = \{s-k \: : \: s \in S\}$.

MacMahon's Multiplication Theorem~\cite[Article 159]{MacMahon}
relates the descent set statistics of two
sets that differ by only one element, stated as
$$   \beta_{n}(S) + \beta_{n}(S \triangle \{k\})
     =
     \binom{n}{k}
        \cdot
     \beta_{k}(S \cap [k-1]) 
        \cdot
     \beta_{n-k}(S \cap [k+1, n-1] - k)  . $$
This result is usually written with the assumption
$k \not\in S$ and the left hand-side as
$\beta_{n}(S) + \beta_{n}(S \cup \{k\})$,
whereas we find it more convenient to work with the symmetric difference.

One way to compute the descent set statistics is via
the flag $f$-vector of the Boolean algebra.
For $S = \{s_{1} < s_{2} < \cdots < s_{k}\} \subseteq [n-1]$,
let $\co(S) = \vec{c} = (c_{1}, c_{2}, \ldots, c_{k+1})$
be the associated composition of $n$
where $c_{i} = s_{i} - s_{i-1}$, where we let $s_{0} = 0$
and $s_{k+1} = n$.
Then the flag $f$-vector of the Boolean algebra~$B_{n}$
is given by the multinomial coefficient
$$  f_{S}  = \binom{n}{\vec{c}} = \binom{n}{c_{1}, c_{2}, \ldots, c_{k+1}}  ,  $$
and the descent set statistics is given by
the inclusion-exclusion
\begin{equation}
  \beta_n(S) = \sum_{T \subseteq S} (-1)^{|S-T|} \cdot f_{T}  . 
\label{equation_inclusion-exclusion}
\end{equation}
An efficient encoding of all the flag $f$-vector entries
of the Boolean algebra is by
the quasi-symmetric function. For a composition
$\vec{c} = (c_{1}, c_{2}, \ldots, c_{k})$
let $M_{\vec{c}}$ denote the monomial quasi-symmetric function
defined by
$$    M_{\vec{c}} 
    =
       \sum_{1 \leq i_{1} < i_{2} < \cdots < i_{k}}
                x_{i_{1}}^{c_{1}} \cdot
                x_{i_{2}}^{c_{2}} \cdots
                x_{i_{k}}^{c_{k}}  .   $$
The algebra of quasi-symmetric functions is the linear span
of the monomial quasi-symmetric functions.
Multiplication of monomial quasi-symmetric functions
is described in Lemma~3.3 in~\cite{Ehrenborg_Hopf}.
Now the quasi-symmetric function of the Boolean algebra
is given in~\cite{Ehrenborg_Hopf} by
$$   F(B_{n})  = (x_{1} + x_{2} + \cdots)^{n}
                      = M_{(1)}^{n}
                      = \sum_{\vec{c}}  \binom{n}{\vec{c}} \cdot M_{\vec{c}} .  $$
The purpose of quasi-symmetric functions is that they allow
efficient computations of the flag $f$-vector modulo a prime $p$,
using the classical relation $(x+y)^{p} \equiv x^{p} + y^{p} \bmod p$.
Finally, using the inclusion-exclusion
equation~\eqref{equation_inclusion-exclusion},
we obtain information about the descent set statistics.
Below is a lemma,
adapted from Lemma 3.2~in~\cite{C_E_P_R},
to compute
the quasi-symmetric function of the Boolean algebra
$F(B_{n})=M_{(1)}^{n}$ modulo a prime.
\begin{lemma}
For $p$ prime and
$n=d_{1} p^{j_{1}}+d_{2} p^{j_{2}}+\cdots+d_{k} p^{j_{k}}$
with $j_{1} > \cdots > j_{k} \geq 0$,
the quasi-symmetric function of the Boolean algebra
$B_{n}$ modulo $p$ is given by
$F(B_{n}) \equiv \prod_{i=1}^k M_{(p^{j_{i}})}^{d_{i}} \bmod p$.
\label{lemma_power_prime}
\end{lemma}
\begin{proof}
The congruence
$(x+y)^{p^{m}} \equiv x^{p^{m}}+y^{p^{m}} \bmod p$
extends to monomial quasi-symmetric functions as
$M_{(1)}^{p^{m}} \equiv M_{(p^{m})} \bmod p$.
Hence
the quasi-symmetric function of Boolean algebra $B_{n}$
is evaluated as follows:
\begin{align*}
F(B_{n})
=
M_{(1)}^{d_{1} p^{j_{1}}+d_{2} p^{j_{2}}+\cdots+d_{k} p^{j_{k}}}
& =
\left( M_{(1)}^{p^{j_{1}}} \right)^{d_{1}}
\cdot
\left( M_{(1)}^{p^{j_{2}}} \right)^{d_{2}}
\cdots
\left( M_{(1)}^{p^{j_{k}}} \right)^{d_{k}}
\\
& \equiv
M_{(p^{j_{1}})}^{d_{1}}
\cdot
M_{(p^{j_{2}})}^{d_{2}}
\cdots
M_{(p^{j_{k}})}^{d_{k}}
\bmod p.
\qedhere
\end{align*}

\end{proof}

Chebikin et al.\ defined essential elements in the
case of base $2$, and we extend this notion to base~$p$
for any prime $p$.

\begin{definition}
Let $p$ be a prime and $1 \leq k \leq n-1$.
We say $k$ is {\em essential} for $n$ in base $p$
if we expand both $n$ and $k$ in base $p$,
that is,
$n = \sum_{i \geq 0} n_{i} \cdot p^{i}$
and
$k = \sum_{i \geq 0} k_{i} \cdot p^{i}$
where $0 \leq k_{i}, n_{i} < p$, and
the inequality $k_{i} \leq n_{i}$ holds for all indices $i$.  Otherwise we say $k$ is {\em non-essential} for $n$ in base $p$.
\end{definition}

A different way to state that $k$ is essential for $n$ in base $p$
is that when adding $k$ and $n-k$ in base~$p$ there are no carries.
Directly from this interpretation we have the following natural symmetry:
\begin{lemma}
The element $k$ is essential for $n$ in base $p$
if and only if
$n-k$ is essential for $n$ in base~$p$.
\label{lemma_essential}
\end{lemma}
Another alternative interpretation is as follows:
\begin{lemma}
\label{lemma_binom}
The element $k$ is essential for $n$ in base $p$
if and only if $\binom{n}{k} \not\equiv 0 \bmod p$.
\end{lemma}
\begin{proof}
By Lucas' theorem, see~\cite[Chapter~XXIII, Section~228]{Lucas},
we have that
$$ \binom{n}{k}
         \equiv
     \prod_{i \geq 0}
       \binom{n_{i}}{k_{i}}
     \bmod p  .  $$
Observe that for $0 \leq k_{i}, n_{i} \leq p-1$
we have that
$\binom{n_{i}}{k_{i}} \not\equiv 0 \bmod p$
if and only if $k_{i} \leq n_{i}$.
\end{proof}

Note that for an element $k$ which is non-essential in base $p$, 
the previous lemma implies that $p$ divides $\binom{n}{k}$.
This allows the following lemma to apply for this number $k$ when
we set the integer $m$ to be the prime $p$.

\begin{lemma}
Let $m$ and $k$ be positive integers such that $1 \leq k \leq n-1$
and $m$ divides $\binom{n}{k}$.
For a subset $S$ of $[n-1]$
the following holds
\begin{align*}
\beta_{n}(S)
& \equiv
-\beta_{n}(S \triangle \{k\}) \bmod m . 
\end{align*}
\label{lemma_flip}
\end{lemma}
\begin{proof}
By MacMahon's multiplication theorem we have that
$$   \beta_{n}(S) + \beta_{n}(S \triangle \{k\})
     =
     \binom{n}{k}
        \cdot
     \beta_{k}(S \cap [k-1]) 
        \cdot
     \beta_{n-k}(S \cap [k+1, n-1] - k)  ,   $$
and the result follows 
by the assumption that
$\binom{n}{k} \equiv 0 \bmod m$.
\end{proof}

For $0 \leq j \leq m-1$ define
$a_{m,j}$ to be the number of subsets $S \subseteq [n-1]$
such that $\beta_{n}(S) \equiv j \bmod m$.
Note that we suppress the dependency on $n$.
Furthermore, if $m$ is clear from the context, we simply
write $a_{j}$. 

\begin{lemma}
\label{lemma_flip_sign}
Let $m$ be a positive integer and $1 \leq k \leq n-1$.
If $m$ divides $\binom{n}{k}$
then the equality $a_{m,j} = a_{m,-j}$ holds for all $j$.
\end{lemma}
\begin{proof}
By Lemma~\ref{lemma_flip}
we have that
$\beta_{n}(S) \equiv -\beta_{n}(S \triangle \{k\}) \bmod m$.
Hence the map sending $S$ to the symmetric difference
$S \triangle \{k\}$ yields a bijection between the sets counted 
by $a_{m,j}$ and~$a_{m,-j}$.
\end{proof}

The following are consequences of Theorem~2.1 in~\cite{C_E_P_R}, which gives information about the proportion of even or odd descent statistics $\beta_{n}(S)$ depending on the number of 1's in the binary expansion of $n$.  We apply their result to achieve equalities involving $a_{i,j}$.

\begin{theorem}[Chebikin et al.]
\label{theorem_binary_exp}
\begin{enumerate}[(a)]
\item If $n$ has only one 1 in its binary expansion, i.e. $n=2^{a}$,
then
$\beta_{n}(S)\equiv 1 \bmod 2$ for all subsets $S\subseteq [n-1]$.

\item If $n$ has either two or three 1's in its binary expansion, then there is an identical number of even descent statistics as there is of odd descent statistics.
\end{enumerate}
\end{theorem}
In terms of the proportion introduced in the introduction,
we have
$\rho(2^{a}) = 1$,
$\rho(2^{b} + 2^{a}) = 1/2$
and
$\rho(2^{c} + 2^{b} + 2^{a}) = 1/2$
for non-negative integers $c > b > a$.
As a direct corollary we have
\begin{corollary}
\label{corollary_binary_exp}
Let $s$ be an odd positive integer.
\begin{enumerate}[(a)]
\item If $n$ has only one $1$ in its binary expansion,
then for $j$ even $a_{2s,j} = 0$ holds.

\item If $n$ has either two or three $1$'s in its binary expansion,
then
$$ \sum_{\substack{j=0 \\ j \text{ even}}}^{2s-2} a_{2s,j}
    = \sum_{\substack{j=0 \\ j \text{ odd}}}^{2s-1} a_{2s,j} . $$
\end{enumerate}
\end{corollary}

We end with a well-known fact from algebra.
\begin{fact}
If $f(t)$ is a polynomial in $\Qqq[t]$
with $e^{{2\pi i}/{j}}$ as a root of
multiplicity $r$ then the $j$th cyclotomic polynomial
$\Phi_{j}(t)$ is a factor of order $r$ of $f(t)$.
\end{fact}
This follows since the cyclotomic polynomial is
the minimal polynomial of
$e^{{2\pi i}/{j}}$
over the rational field~$\Qqq$.

\section{The simplicial complex $\Delta_{n}$}
\label{section_complex}

We now introduce a simplicial complex, which will
encode the descent set statistics modulo $2$,
via the reduced Euler characteristic.
Let $\Delta_{n}$ be a simplicial complex on the vertex set $[n-1]$.
Let $F$ be a face of
$\Delta_{n}$ if when adding the entries of the associated composition
$\co(F) = (c_{1}, c_{2}, \ldots, c_{k+1})$,
that is, the sum $c_{1} + c_{2} + \cdots + c_{k+1} = n$
has no carries in base $2$.

Notice that $\{i\}$ is a vertex of $\Delta_{n}$ if and only
if $i$ is an essential element of $n$ in base $2$.
In fact, the simplicial complex $\Delta_{n}$ is completely described
by the number of $1$'s in the binary expansion of~$n$.
For $n$ with $k$ $1$'s in its binary expansion, the complex
$\Delta_{n}$ is the barycentric subdivision of
the boundary of a $(k-1)$-dimensional simplex.
A different way to describe it
is that $\Delta_{n}$ is the boundary of the dual of the
$(k-1)$-dimensional permutahedron.

\begin{theorem}
The quasi-symmetric function of $B_{n}$ modulo $2$, is given
by
$$   F(B_{n}) 
   \equiv
     \sum_{F \in \Delta_{n}}
            M_{\co(F)}
   \bmod 2 .  $$
\end{theorem}
\begin{proof}
Write $n$ as a sum of 2-powers, that is,
$n = 2^{j_{1}} + 2^{j_{2}} + \cdots + 2^{j_{k}}$
where $j_{1} > j_{2} > \cdots > j_{k}$.
By Lemma~\ref{lemma_power_prime}
we have have the identity
$$
F(B_{n}) \equiv
M_{2^{j_{1}}} \cdot M_{2^{j_{2}}} \cdots M_{2^{j_{k}}} 
\bmod 2.
$$
Now when multiplying out these $k$ monomial quasi-symmetric functions
we obtain a sum over monomial quasi-symmetric functions, where the
indexing composition has parts consisting of sums of the $2$-powers
$2^{j_{1}}$, $2^{j_{2}}$, \ldots, $2^{j_{k}}$.
Furthermore, each $2$-power can only appear in exactly one part
and only once in that part.
Also note no composition can be created in two different ways;
in the language of the article~\cite{Ehrenborg_Readdy},
the partition
$\{2^{j_{1}}, 2^{j_{2}}, \ldots, 2^{j_{k}}\}$ is a knapsack partition.
Finally, translating the compositions of $n$ into subsets of $[n-1]$
proves the result.
\end{proof}

In other words, the flag $f$-vector entry $f_{S}(B_{n})$ is odd
if and only if
$S$ is a face of the complex $\Delta_{n}$.
Let $\Delta_{n}\vrule_{S}$ denote the simplicial complex
$\Delta_{n}$ restricted to vertex set $S$, that is,
$$  \Delta_{n}\vrule_{S} 
   =
    \{F \subseteq S \: : \: F \in \Delta_{n}\}  .  $$

\begin{theorem}
The descent set statistic
$\beta_{n}(S)$ modulo $2$ is given by
the reduced Euler characteristic of the
induced subcomplex $\Delta_{n}\vrule_{S}$,
that is,
$$  \beta_{n}(S) \equiv \widetilde{\chi}(\Delta_{n}\vrule_{S}) \bmod 2. $$
\label{theorem_Euler}
\end{theorem}
\begin{proof}
By a direct computation
\begin{align*}
\beta_{n}(S)
 & \equiv
\sum_{T \subseteq S} (-1)^{|S-T|} \cdot f_{T}(B_{n}) \\
 & \equiv
\sum_{T \subseteq S} (-1)^{|T|-1} \cdot f_{T}(B_{n}) \\
 & \equiv
\sum_{T \subseteq S, \: T \in \Delta_{n}} (-1)^{|T|-1} \\
 & \equiv
\widetilde{\chi}(\Delta_{n}\vrule_{S}) \bmod 2 . \qedhere
\end{align*}
\end{proof}

\section{One binary digit}
\label{section_one_digit}

In this section we explore cyclotomic factors
in the descent set polynomial $Q_{n}(t)$
where $n$ is a power of $2$, that is, $n$ has
one $1$ in its binary expansion.
First we have a result showing conditions on the
values of $a_{m,j}$ when we have
a cyclotomic factor in the general $n$th descent
set polynomial.
Note that we abbreviate $a_{m,j}$ as $a_{j}$.

\begin{lemma}
\label{factor_lemma}
Let $m$ be an even positive integer.
The cyclotomic polynomial $\Phi_{m}$ is a factor of the descent set polynomial $Q_{n}(t)$ if the following equations hold:
\begin{align}
a_{j} & = a_{-j} , 
\label{equation_vertical} \\
a_{j} & = a_{m/2-j} ,
\label{equation_horizontal} 
\end{align}
for all integers $j$.
\end{lemma}
\begin{proof}
Consider the primitive $m$th root of unity
$\omega=e^{i \pi\slash m}$.  In order for $\Phi_{m}$ to be a factor of 
$Q_{n}(t)$, we must have $Q_{n}(\omega)=0$.
Since $\omega^{m} = 1$, we need to show 
$$ Q_{n}(\omega)
   =
\sum_{S\subseteq[n-1]} \omega^{\beta_n(S)}
   =
   a_{0} + a_{1} \cdot \omega + a_{2} \cdot \omega^{2} +
          \cdots +a_{m-1} \cdot \omega^{m-1}$$
is zero.
By reflection in the real and the imaginary axis in
the complex plane we have
$\omega^{-j} + \omega^{j} + \omega^{m/2-j} + \omega^{m/2+j} = 0$,
from which the result follows.
\end{proof}

Assume that $s$ is an odd positive integer.
We consider which values of $s$ such that
the $4s$th cyclotomic polynomial, $\Phi_{4s}$, divides the descent set polynomial $Q_{n}(t)$
when $n$ is a power of $2$.

\begin{theorem}
\label{theorem_one_digit}
Let $n=2^{a}$ where $a \geq 2$.
Assume that $s$ is an odd integer such that
$s$ divides the central binomial coefficient
$\binom{n}{n/2}$
and
$s$ divides $\binom{n}{k}$ for some $k \neq n/2$.
Then the cyclotomic polynomial $\Phi_{4s}(t)$
divides the descent set polynomial $Q_{n}(t)$.
\end{theorem}
\begin{proof}
Observe that there is one carry in the addition
$n/2 + n/2 = n$ in base $2$. Hence by Kummer's
theorem,
see~\cite[Pages 115--116]{Kummer},
$2$ is the largest $2$-power dividing
$\binom{n}{n/2}$. In other words,
$\binom{n}{n/2} \equiv 2 \bmod 4$.
Combining this with the fact that $s$ divides this central
binomial coefficient, we have
$\binom{n}{n/2} \equiv 2s \bmod 4s$.
Thus, MacMahon's multiplication theorem gives that
\begin{align*}
\beta_{n}(S)+\beta_{n}(S \triangle \{n/2\})
&=
\binom{n}{n/2}
\cdot
\beta_{n/2}(S \cap [1,n/2-1])
\cdot
\beta_{n/2}(S \cap [n/2+1, n-1]-n/2) .
\end{align*}
Since $\beta_{n/2}$ only takes odd values as shown in Theorem~\ref{theorem_binary_exp}(a), we obtain
that
\begin{align*}
\beta_{n}(S)+\beta_{n}(S \triangle \{n/2\})
& \equiv
2  s \bmod 4  s.
\end{align*}
Thus, the statement $\beta_{n}(S) \equiv j \bmod 4s$ is equivalent to 
$\beta_{n}(S \triangle \{n/2\}) \equiv 2s-j \bmod 4s$. 
That is,  the map
$S \longmapsto S \triangle \{n/2\}$
yields a bijection that proves $a_{j}=a_{2s-j}$
for all $j$.

Next since the addition $k + (n-k) = n$ in base $2$ has
at least two carries, we obtain that $2^{2} = 4$ divides the
binomial coefficient $\binom{n}{k}$.
Hence, $4  s$ divides $\binom{n}{k}$ 
and by Lemma~\ref{lemma_flip_sign}
the equality $a_{j} = a_{-j}$ holds for all $j$.
We now have that both equations~\eqref{equation_vertical}
and~\eqref{equation_horizontal}
from Lemma~\ref{factor_lemma} are upheld; thus,
the cyclotomic polynomial $\Phi_{4s}$ divides $Q_{n}(t)$.
\end{proof}

\begin{remark}
{\rm
The case $n = 32 = 2^{5}$ and $k=15$ is particularly nice.
We have that
$\binom{32}{16}
  = 2 \cdot 3^{2} \cdot  5 \cdot 17 \cdot 19 \cdot 23 \cdot 29 \cdot 31$
and
$\binom{32}{15} = 16/17 \cdot \binom{32}{16}$.
Hence, for any divisor $s$ of
$3^{2} \cdot  5 \cdot 19 \cdot 23 \cdot 29 \cdot 31$
and there are $96$ such divisors,
we obtain the cyclotomic factor $\Phi_{4s}$
of $Q_{32}(t)$.
Furthermore, we do not obtain any more cyclotomic factors
by changing $k$,
that is, all the the odd divisors of $\binom{32}{k}$ for $k \leq 14$
are divisors of $\binom{32}{15}$.
}
\label{remark_32}
\end{remark}

\begin{table}
\begin{center}
  \begin{tabular}{ | c | c | c | c | l | l |}
  \hline
$n$ & $s$ & $k$ & Chebikin et & Our \\ 
&&& al.\ statement & statement \\ \hline
    \hline
4 & 1 & 1 & Thm.\ 3.5 & Thm.~\ref{theorem_one_digit}   \\ \hline
8 & 1 & 1 & Thm.\ 3.5 & Thm.~\ref{theorem_one_digit}   \\ \hline
8 & 7 & 2 & & Thm.~\ref{theorem_one_digit}   \\ \hline
16 & 1 & 1 & Thm.\ 3.5 & Thm.~\ref{theorem_one_digit}  \\ \hline
16 & 5, 11, 13, 55 & 7 & & Thm.~\ref{theorem_one_digit}  \\
& 65, 143, 715 && & \\ \hline
16 & 3, 15 & 5 & & Thm.~\ref{theorem_one_digit}  \\ \hline
16 & 39  & 2 & & Thm.~\ref{theorem_one_digit}  \\ \hline
32 & all the divisors & 15 & & Rem.~\ref{remark_32} \\
     & of 17678835 & & & \\ \hline
\end{tabular}
\end{center}
\caption{Examples of cyclotomic factors of 
$Q_{n}(t)$ of the form $\Phi_{4s}$ where $n=2^{a}$.}
\label{table_factors_1}
\end{table}

See Table~\ref{table_factors_1} for examples of cyclotomic factors of $Q_{2^a}(t)$ 
that are explained by Theorem~\ref{theorem_one_digit}, along the $k$ value in 
which $s$ divides $\binom{2^a}{k}$.

\section{Two binary digits}
\label{section_two_digits}

Now we state the result that lets us deduce
cases when the cyclotomic polynomial
$\Phi_{2s}$, where $s$ is an odd positive integer,
 is a factor of the descent set polynomial
$Q_{n}(t)$ when $n$ has 
two 1's in its binary expansion.

\begin{theorem}
\label{theorem_two_digits}
Let $n=2^{b}+2^{a}$, where $b>a$
and $s$ is an odd positive integer.
Assume that $s$ divides~$\binom{n}{2^{a}}$.
Furthermore, assume there is an integer $k$
which is non-essential in base $2$
(that is, $k \neq 2^{a}, 2^{b}$)
and such that $s$ divides $\binom{n}{k}$.
Then the cyclotomic polynomial $\Phi_{2s}$
is a factor of $Q_{n}(t)$.
\end{theorem}
\begin{proof}
Since there are no carries in the addition
$2^{b} + 2^{a} = n$ in base $2$, by Kummer's theorem
we know that $\binom{n}{2^{a}}$ is odd.
Combining this fact with the congruence modulo $s$,
we obtain $\binom{n}{2^{a}} \equiv s \bmod 2 s$.
Therefore, by MacMahon's multiplication theorem,
we have that 
\begin{align}
\beta_{n}(S) + \beta_{n}(S \triangle \{2^{a}\})
   &=
\binom{n}{2^{a}}
  \cdot
\beta_{2^{a}}(S \cap [2^{a}-1])
  \cdot
\beta_{2^{b}}(S \cap [2^{a}+1, n-1]-2^{a})  \notag \\
  & \equiv
s \bmod 2  s,
\label{equation_2a_2b}
\end{align}
since both $\beta_{2^{a}}$ and $\beta_{2^{b}}$
are odd.
Hence, we use the bijective map $S\longmapsto S\triangle \{2^a\}$ to 
conclude that $a_{j} = a_{s-j}$ for all $j$.

Since the addition
$k + (n-k)$ has at least one carry in base $2$
the binomial coefficient $\binom{n}{k}$ is even.
Hence $\binom{n}{k}$ is divisible by $2  s$.
By Lemma~\ref{lemma_flip_sign}
the inequality $a_{j} = a_{-j}$ holds for all $j$.
Combining these two equalities using
Lemma~\ref{factor_lemma},
the result follows.
\end{proof}

We begin by two remarkable examples.
\begin{remark}
{\rm
Consider the case $n=18 = 2^{4} + 2^{1}$ and $k=4$.
Note that $\binom{18}{2} = 3^{2} \cdot 17 = 153$.
Furthermore note that
$\binom{18}{4} = 2^{2} \cdot 5 \cdot \binom{18}{2}$.
Hence for any divisor $s$ of $153$
we obtain that the cyclotomic polynomial $\Phi_{2s}$
divides the descent set polynomial $Q_{18}(t)$.
This argument explains all
the cyclotomic factors found in the descent set
polynomial $Q_{18}(t)$; see Table~\ref{table_P}.
}
\label{remark_18}
\end{remark}

\begin{remark}
{\rm
Consider the case $n=20 = 2^{4} + 2^{2}$ and $k=6$.
Now we have $\binom{20}{4} = 3 \cdot 5 \cdot 17 \cdot 19 = 4845$
and
$\binom{20}{6} = 2^{3} \cdot \binom{20}{4}$.
Hence for any divisor $s$ of $4845$
the cyclotomic polynomial $\Phi_{2s}$
is a factor in the descent set polynomial $Q_{20}(t)$,
explaining all the $16$ known cyclotomic factors;
see the longest row in Table~\ref{table_P}.  
}
\label{remark_20}
\end{remark}

We now continue to study the case when the
integer $s$ is an odd prime $p$.  Recall from
Lemma~\ref{lemma_binom} that $k$ being a non-essential
element in base $p$ implies that $p$ divides $\binom{n}{k}$.
Hence, to satisfy the assumptions in Theorem~\ref{theorem_two_digits} 
for this case, we need to show that $2^a$ and $k$ are 
non-essential in base~$p$ and that $k$ is non-essential in base 2.

Note however that for two relative prime integers $p$ and $q$,
a carry in the addition $k + (n-k) = n$ in base $p \cdot q$
does not imply a carry for this addition in both base $p$ and $q$.
An example the addition $12+3 = 15$.
In base $15$ there is a carry, where as in base $3$ there
is no carry.

The following lemma is useful in determining when $2^{a}$ is non-essential for $n$ in base $p$, where $p$ is prime,
in order to apply Theorem~\ref{theorem_two_digits}.  Although rarely cited during the subsequent arguments
since we often need the actual value of $i+j \bmod p$ instead of only the fact that it is at least $p$, it provides
reasoning for finding particular values of $n$.

\begin{lemma}
For $n=2^{a}+2^{b}$,
if $2^{a}\equiv i \bmod p$ and $2^{b} \equiv j \bmod p$ where 
$1\leq i,j\leq p-1$ and $i+j\geq p$,
then $2^{a}$ is non-essential for $n$ in base $p$.
\label{lemma_non-essential}
\end{lemma}
\begin{proof}
Since $i,j\leq p-1$ and $i+j\geq p$, we have $i > i+j \bmod p$.  Therefore, the last digit of the
base $p$ expansion of $2^{a}$ is larger than the last digit of the base $p$ expansion of $n$,
causing $2^{a}$ to be non-essential in base $p$.
\end{proof}

The following theorems provide conditions for the prime $p$,
the multiplicative order $g$ of~$2$ in $\Zzz_{p}^{*}$,
and the exponents $a$
and~$b$ that allow Theorem~\ref{theorem_two_digits} to be applied to show that $\Phi_{2p}$ is a factor of~$Q_{n}(t)$.

\begin{theorem}
\label{theorem_even_order_one}
Assume that 2 has order $g$ in the multiplicative group $\Zzz_{p}^{*}$
where $g$ is even.  Let $n=2^{b}+2^{a}$ where we assume $b>a$ and $n\geq 9$.
If we have $\{a,b\}\equiv \{0,{g}/{2}\} \bmod g$,
then $2^{a}$ is non-essential in base $p$.
Furthermore, the element $7$ is non-essential for
both base $2$ and base $p$.
Hence~$\Phi_{2p}$ is a factor of~$Q_{n}(t)$.
\end{theorem}
\begin{proof}
Since
$2^{g/2} \not\equiv 1 \bmod p$
and
$(2^{g/2} - 1) \cdot (2^{g/2} + 1) = 2^{g} - 1 \equiv 0 \bmod p$
we know that
$2^{g/2}\equiv -1 \bmod p$ using that $p$ is a prime.
Hence the last digits of $2^a$ and $2^b$ in their
base $p$ expansions are $1$ and $p-1$, in some order.
Thus, we have 
$n=2^{b}+2^{a} \equiv 1+(p-1) \equiv 0 \bmod p$,
that is, the last digit in the base $p$ expansion of~$n$ is 0.
Hence $2^{a}$ is non-essential
in base $p$.

Notice that $7$ has three non-zero digits in its binary expansion compared 
to only $2$ such digits for~$n$,
making $7$ non-essential for $n$ in base $2$.
Since the order of $2$ in $\Zzz_{7}^{*}$ is $3$, which is odd,
we have $p \neq 7$.
Finally,
the last digit of the base $p$ expansion of~$7$ is non-zero
for all odd primes $p \neq 7$.
Hence $7$ is also non-essential for $n$ in base $p$,
completing the result.
\end{proof}

\begin{remark}
{\rm
The assumption in Theorem~\ref{theorem_even_order_one} of $n\geq 9$ was needed in order for $7$ 
to always be a non-essential element, but note that the theorem can still be applied when $n=6$ if $p=3$.  
The element~$5$ is instead chosen as the non-essential element in base 2 and in base $p$.
}
\label{remark_one}
\end{remark}

\begin{theorem}
\label{theorem_even_order_two}
Assume that 2 has order $g$ in the multiplicative group $\Zzz_{p}^{*}$
where $g$ is even.  Let $n=2^{b}+2^{a}$ where we assume $b>a$ and $n>2p-1$.
If we have $a\equiv b\equiv {g}/{2} \bmod g$,
then $2^{a}$ is non-essential in base $p$.  Furthermore, the element $2p-1$ is
non-essential for both base $2$ and base $p$.  
Hence~$\Phi_{2p}$ is a factor of~$Q_{n}(t)$.
\end{theorem}
\begin{proof}
Similar to part of the previous proof, we have in this case that 
$2^{a}\equiv 2^{b}\equiv 2^{{g}/{2}}\equiv p-1 \bmod p$.  Therefore, 
$n=2^{a}+2^{b}\equiv (p-1)+(p-1)\equiv p-2 \bmod p$.
Thus, the last digit of the base $p$ expansion of $n$ is $p-2$ while the last digit
of the expansion of $2^{a}$ is $p-1$, making $2^{a}$ be non-essential in base $p$.

Since $2p-1$ is odd, the last digit in its base~$2$ expansion is~$1$,
but the last digit of the base~$2$
expansion of $n$ is $0$ because $a,b \neq 0$.
Hence $2p-1$ is non-essential in base $2$. 
Additionally,
$2p-1\equiv p-1 > p-2 \bmod p$,
thus it is non-essential in base $p$ as well.
\end{proof}

\begin{remark}
{\rm
The equivalence conditions on the exponents
within Theorems~\ref{theorem_even_order_one} 
and~\ref{theorem_even_order_two} are not the only such
conditions that makes the theorem hold true when $g$ is even. 
These are many such conditions, especially if $2$
is a generator of $\Zzz_{p}^{*}$ 
since the powers of $2$ contain every 
possible non-zero value as the last digit,
and all that is needed is for the argument in the 
proof of $2^{a}$ being non-essential in base $p$
is for the sum of these digits 
to be at least $p$, as shown in Lemma~\ref{lemma_non-essential}.
In this case of 
$2$ being a generator of $\Zzz_{p}^{*}$ for $p=2r+1$, there are exactly 
$r\cdot (r+1)$ of pairs of possible exponents modulo~$g$ that will work.
One still needs to find element $k$ that is non-essential 
in base $2$ and in base~$p$.  Finding this~$k$ value is easy if given 
a particular pair of $n$ and $p$ values, but this step causes a further generalization 
of the proof to be difficult.

Table~\ref{equivalencies_table} includes all of
the equivalence conditions modulo the order $g$ for four 
odd primes that lead to $2^{a}$ being non-essential for $n$ in base $p$.  
For examples of finding the non-essential $k$ value,
see Table~\ref{table_factors_2}.
}
\label{remark_two}
\end{remark}

\begin{table}
\begin{center}
\begin{tabular}{ | c | c | c | }
  \hline
   $p$ & $g$ & $\{a,b\} \bmod g$ \\ \hline
    \hline
    3 & 2 & $\{0,1\}, \{1,1\}$   \\ \hline
    5 & 4 & $\{0,2\}, \{1,2\},\{1,3\},\{2,2\},\{2,3\},\{3,3\}$ \\ \hline
    11 & 10 & $\{0,5\},\{1,5\},\{1,6\},\{2,3\},\{2,5\},\{2,6\},\{2,7\},\{3,4\},\{3,3\},\{3,5\},$ \\ & & $\{3,6\},\{3,7\},
   \{3,8\},\{3,9\},\{4,5\},\{4,6\},\{4,7\},\{4,9\},\{5,5\},\{5,6\},$ \\ & & $\{5,7\},\{5,8\},\{5,9\},\{6,6\},\{6,7\},\{6,8\},\{6,9\},\{7,7\},
    \{7,9\},\{9,9\}$ \\ \hline
    17 & 8 & $\{0,4\},\{1,4\},\{1,5\},\{2,4\},\{2,5\},\{2,6\},\{3,4\},\{3,5\}, \{3,6\},\{3,7\},$ \\ & & $\{4,4\},\{4,5\},\{4,6\},\{4,7\},
    \{5,5\},\{5,6\},\{5,7\},\{6,6\},\{6,7\},\{7,7\}$ \\ \hline
\end{tabular}
\end{center}
\caption{Examples of equivalency conditions for small prime numbers.}
\label{equivalencies_table}
\end{table}

\begin{remark}
\label{remark_three}
{\rm
If $2$ has multiplicative order $g$ in $\Zzz_{p}^{*}$,
then its order $G$ in $\Zzz_{p^{\hspace{.03cm}l}}^{*}$
is a divisor of $p^{\hspace{.05cm}l-1} \cdot g$.
The order $g$ gives the length of the repeating sequence 
of the last digit of the base $p$ expansions of the powers $2^{a}$,
and likewise, the order~$G$ gives
the length of the repeating sequence of
the last $l$ digits of those powers of $2$. 
Similar reasoning to Lemma~\ref{lemma_non-essential}
applies when adding together any pair of
digits together, not just the last digit.
Thus, there are equivalencies modulo~$G$
that cause $2^{a}$
to be non-essential in base $p$ because of a carry in 
one of the last $l$ digits.  As an example, when $p=3$ 
the order of $2$ in~$\Zzz_{9}^{*}$ is~$6$, hence
the last two digits of $2^{a}$ cycle through the six values 01, 02, 11, 22, 21 and 12 as $a$ increases.  
Therefore, when $\{a,b\}\equiv \{2,4\}\bmod 6$, the last two digits of $n$ in base 3 are $11+21\equiv 02$, so the
second digit from the right is larger for $2^{a}$ than for $n$, making it non-essential in base $p$.
}
\end{remark}

\begin{theorem}
\label{theorem_case_two}
Let $n=2^{b}+2^{a}$ where we assume $b>a$ and $n\geq 5$, and also
assume that $p>3$.
If we have $a,b\equiv g-1 \bmod g$
where $g$ is the multiplicative order of~$2$,
then $2^{a}$ is non-essential in base $p$.  Furthermore, the element $3$ 
is non-essential in both base 2 and base $p$.
Hence~$\Phi_{2p}$ is a factor of~$Q_{n}(t)$.
\end{theorem}
\begin{proof}
If the multiplicative order of $2$ is $g$, 
then $g$ is the smallest integer 
so that $2^{g}\equiv 1 \bmod p$.
Thus, $2^{a}\equiv 2^{b}\equiv 2^{g-1} > 1 \bmod p$, and 
$n=2^{a}+2^{b}\equiv 2^{g-1}+2^{g-1}\equiv 2^{g}\equiv 1\bmod p$.
Hence $2^{a}$ is non-essential in base $p$
because the last digit in its base $p$ expansion is larger than that of $n$.

The element $3$ is non-essential for $n$ in base $2$ since our assumption 
of $n\geq 5$ implies that $b\geq 2$.  Because of our assumption that $p>3$, 
the element $3$ is also non-essential
in base $p$ since the last digit of the base~$p$ expansion
for $n$ is $1<3$, concluding the result.
\end{proof}

Note that we omitted $p=3$ from the previous result because this was already 
proven for $p=3$ in Theorem~\ref{theorem_even_order_two} due to the order
of $2$ being $g=2$, making ${g}/{2}=g-1$. 

\begin{remark}
\label{remark_Mersenne}
{\rm
Assuming $p>3$, if $p$ is a Mersenne prime,
that is, $p$ has the form $2^{q}-1$
implying that $q$ is also a prime number, the equivalence
condition on the exponents in Theorem~\ref{theorem_case_two}
is the only such condition modulo~$g$ for which
$\Phi_{2p}$ is a factor of~$Q_{n}(t)$.
The first examples of Mersenne primes after~$3$
are $p=7$, $31$ and~$127$.
}
\end{remark}

\begin{table}
\begin{center}
  \begin{tabular}{ | c | c | c | c | l | l |}
  \hline
$n$ & $s$ & $\{a,b\}\bmod g$ & $k$ & Chebikin et & Our \\ 
&&&& al.\ statement & statement \\ \hline
    \hline
    6 & 3 & $\{0,1\}$ & 5 & 
    Thm.~5.6 & Rem.~\ref{remark_one}   \\ \hline
    6 & 5 & $\{1,2\}$ & 3
    && Rem.~\ref{remark_two}  \\ \hline
    9 & 3 & $\{0,1\}$ & 7 &
    Thm.~5.5 & Thm.~\ref{theorem_even_order_one} \\ \hline
    9 & 9 && 2 &
    & Thm.~\ref{theorem_two_digits} \\ \hline
    10 & 3 & $\{1,1\}$ & 5
    && Thm.~\ref{theorem_even_order_two} \\ \hline
    10 & 5 & $\{1,3\}$ & 1 & 
    Thm.~5.6 & Rem.~\ref{remark_two}  \\ \hline
    10 & 9 && 5 & 
    &  Thm.~\ref{theorem_two_digits} \\ \hline
    10 & 15 && 3 & 
    &  Thm.~\ref{theorem_two_digits} \\ \hline
    12 & 3 & $\{0,1\}$ & 7 
    && Thm.~\ref{theorem_even_order_one} \\ \hline
    12 & 5 & $\{2,3\}$ & 3 
    && Rem.~\ref{remark_two} \\ \hline
    12 & 11 & $\{2,3\}$ & 2 
    && Rem.~\ref{remark_two}  \\ \hline
    12 & 55 && 3 
    && Thm.~\ref{theorem_two_digits} \\ \hline
    12 & 9, 33, 99 && 5 
    && Thm.~\ref{theorem_two_digits} \\ \hline
    17 & 17 & $\{0,4\}$ & 7 & 
    Thm.~5.5 & Thm.~\ref{theorem_even_order_one} \\ \hline
    18 & 17 & $\{1,4\}$ & 3
    && Rem.~\ref{remark_two} \\ \hline
    18 & 9, 51, 153 && 4 
    && Rem.~\ref{remark_18} \\ \hline
    20 & 3 & $\{2,4\} \bmod 6$ & 3 
    && Rem.~\ref{remark_three}  \\ \hline
    20 & 5 & $\{0,2\}$ & 7 
    && Thm.~\ref{theorem_even_order_one} \\ \hline
    20 & 17 & $\{2,4\}$ & 5 
    && Rem.~\ref{remark_two} \\ \hline
    & 15, 19, 51, 57, 85, &&
    && \\
    20 & 95, 255, 285, 323, && 6
    && Rem.~\ref{remark_20} \\
    & 969, 1615, 4845 &&
    && \\ \hline \hline
    72 & 3 & $\{0,1\}$ & 7 
    && Thm.~\ref{theorem_even_order_one} \\ \hline
    528 & 31 & $\{4,4\}$ & 3 
    && Thm.~\ref{theorem_case_two} \\ \hline
    1088 & 5 & $\{2,2\}$ & 9 
    && Thm.~\ref{theorem_even_order_two} \\ \hline
\end{tabular}
\end{center}
\caption{Examples of cyclotomic factors of 
$Q_{n}(t)$ of the form $\Phi_{2s}$,
where the binary expansion of $n$ has two $1$'s.}
\label{table_factors_2}
\end{table}

Table~\ref{table_factors_2} summarizes particular
values of $n$ and $s$ for which $\Phi_{2s}$ is a factor of 
$Q_{n}(t)$ with $n$ having two binary digits. 
An element $k$ that is non-essential in base~$2$ and base~$p$
and the statement explaining why it is a factor are also included.  
For the cases in which $s$ is a prime $p$, the set of exponents modulo the
multiplicative order $g$ of $2$ is also listed. 
The top portion includes factors that were known by Chebikin
et al., although many were left unexplained in their work.
Cases that were proven by Chebikin et al.\ are included
within the statement column. 
The bottom portion displays just a few examples of factors 
that are explained by our results that were previously unknown.  

\section{Three binary digits}
\label{section_three_digits}

We now continue to explore cyclotomic factors $\Phi_{2s}$,
where $s$ is an odd positive integer, in the descent set polynomial
$Q_{n}(t)$ where $n$ has three 1's in its binary expansion.

\begin{theorem}
Let $n=2^{c}+2^{b}+2^{a}$ where $c>b>a$
and $s$ is an odd positive integer.
Assume that 
$s$ divides
the three binomial coefficients
$\binom{n}{2^{a}}$,
$\binom{n}{2^{b}}$ and
$\binom{n}{2^{c}}$.
Assume furthermore that there is an element $k$
which is non-essential in base $2$,
that is,
$k \not\in \{2^{a}, 2^{b}, 2^{a}+2^{b}, 2^{c}, 2^{c}+2^{a}, 2^{c}+2^{b}\}$,
such that
$s$ divides~$\binom{n}{k}$.
Then the cyclotomic polynomial $\Phi_{2s}$
is a factor of the descent set polynomial $Q_{n}(t)$.
\label{theorem_three_digits}
\end{theorem}
\begin{proof}
Since there is an element $k$
which is non-essential in base $2$, we know that
$2$ divides $\binom{n}{k}$.
Thus $2  s$ divides $\binom{n}{k}$,
and  Lemma~\ref{lemma_flip_sign} gives that $a_{j} = a_{-j}$ for all $j$.
Next our major goal is to show that
$a_{j} = a_{s-j}$.
We do that by constructing an involution
$\phi$ on all subsets of $[n-1]$ such that
$\beta_{n}(S) + \beta_{n}(\phi(S)) \equiv s \bmod 2s$.
Hence for every contribution to $a_{j}$
there is a corresponding contribution
to $a_{s-j}$.
The form of the involution $\phi$ will be
$\phi(S) = S \triangle X$ where the
subset $X$ depends on how $S$ intersects the
four element set
$\{2^{a}, 2^{b}, 2^{c}+2^{a}, 2^{c}+2^{b}\}$.

Since the elements $2^{c}$ and $2^{b}+2^{a}$ are both essential in base $2$, we apply 
MacMahon's theorem to get
\begin{align*}
\beta_{n}(S)+\beta_{n}(S \triangle \{2^{b}+2^{a}\})
&=
\binom{n}{2^{b}+2^{a}}
\cdot
\beta_{2^{b}+2^{a}}(S\cap[1,2^{b}+2^{a}-1]) \notag \\
&
\hspace*{23 mm}
\cdot
\beta_{2^{c}}(S\cap[2^{b}+2^{a}+1, n-1]-(2^{b}+2^{a}))
\notag \\
&\equiv \begin{cases}
       0 & \text{if } |S\cap\{2^{a},2^{b}\}|=1, \\
       1 & \text{if } |S\cap\{2^{a},2^{b}\}|=0 \text{ or } 2
           \end{cases}
\:\:\:\:
          \bmod 2 ,\\
\intertext{}
\beta_{n}(S)+\beta_{n}(S \triangle \{2^{c}\})
&=
\binom{n}{2^{c}}
\cdot
\beta_{2^{c}}(S\cap[2^{c}-1])
\cdot
\beta_{2^{b}+2^{a}}(S\cap[2^{c}+1,n-1]-2^{c})
\notag \\
&\equiv
\begin{cases}
       0 & \text{if } |S\cap\{2^{c}+2^{a},2^{c}+2^{b}\}|=1, \\
       1 & \text{if } |S\cap\{2^{c}+2^{a},2^{c}+2^{b}\}|=0 \text{ or } 2
\end{cases}
\:\:\:\:
          \bmod 2 ,
\end{align*}
since the two binomial coefficients
$\binom{n}{2^{b}+2^{a}} = \binom{n}{2^{c}}$ are both odd
and the descent set statistics involving
$\beta_{2^{c}}$
are also odd by
Theorem~\ref{theorem_binary_exp} (a).
Therefore, the sums of these descent
set statistics are determined by the values for $\beta_{2^{b}+2^{a}}$,
which we examine 
by considering the complex $\Delta_{2^{b}+2^{a}}$ and using
Theorem~\ref{theorem_Euler}.  
This complex consists of only of 
two isolated vertices at $2^{b}$ and $2^{a}$.  Thus, the induced subcomplex 
$\Delta_n|_{S\cap [1,2^{b}+2^{a}-1]}$ is a single vertex if $|S\cap\{2^{a},2^{b}\}|=1$ with a
reduced Euler characteristic of 0. Otherwise, it is two isolated vertices or the empty complex, 
both of which have a reduced Euler characteristic of $1 \bmod 2$. The reasoning behind the
second sum is identical once the set $S$ is shifted down by $2^{c}$.

Since $s$ divides
$\binom{n}{2^{c}} = \binom{n}{2^{a}+2^{b}}$,
we have by Lemma~\ref{lemma_flip} that
$\beta_{n}(S) + \beta_{n}(S\triangle\{2^{b}+2^{a}\})
     \equiv
  \beta_{n}(S) + \beta_{n}(S\triangle \{2^{c}\}) \equiv 0 \bmod s$.
Combining this with the modulo $2$ sums,
we have the following results modulo~$2s$
\begin{align}
\beta_{n}(S)+\beta_{n}(S \triangle \{2^{b}+2^{a}\})
&\equiv \begin{cases}
       0 & \text{if } |S\cap\{2^{a},2^{b}\}|=1, \\
       s & \text{if } |S\cap\{2^{a},2^{b}\}|=0 \text{ or } 2
           \end{cases}
\:\:\:\:
          \bmod 2s ,
\label{equation_a_new} \\      
\beta_{n}(S)+\beta_{n}(S \triangle \{2^{c}\})
&\equiv
\begin{cases}
       0 & \text{if } |S\cap\{2^{c}+2^{a},2^{c}+2^{b}\}|=1, \\
       s & \text{if } |S\cap\{2^{c}+2^{a},2^{c}+2^{b}\}|=0 \text{ or } 2
\end{cases}
\:\:\:\:
          \bmod 2s.
\label{equation_b_new}          
\end{align}

We now begin to construct the involution $\phi$.
Assume that
$|S \cap \{2^{a},2^{b}\}| = 0$ or $2$.
Then by equation~\eqref{equation_a_new}, 
$\beta_{n}(S) + \beta_{n}(S \triangle \{2^{b}+2^{a}\})
     \equiv
s \bmod 2s$.
Hence in this case let the involution be given by
$\phi(S) = S \triangle \{2^{b}+2^{a}\}$.

The symmetric case is as follows.
Assume that we have $|S \cap \{2^{a},2^{b}\}| = 1$
and $|S\cap \{2^{c}+2^{a},2^{c}+2^{b}\}| = 0$ or~$2$.
By equation~\eqref{equation_b_new}, 
$\beta_{n}(S) + \beta_{n}(S \triangle \{2^{c}\})
     \equiv
s \bmod 2s$
and let the involution be given by
$\phi(S) = S \triangle \{2^{c}\}$.

The case that remains is when
the set $S$ satisfies
$|S \cap \{2^{a},2^{b}\}| = 1$
and 
$|S \cap \{2^{c}+2^{a},2^{c}+2^{b}\}| = 1$.
By equations~\eqref{equation_a_new}          
and~\eqref{equation_b_new}
we have that      
$$   \beta_{n}(S) \equiv
       \beta_{n}(S \triangle \{2^{b}+2^{a},2^{c}\}) \equiv
       - \beta_{n}(S \triangle \{2^{b}+2^{a}\}) \equiv
       - \beta_{n}(S \triangle \{2^{c}\})
           \bmod 2s   .   $$
Especially,
these four descent set statistics all have the same parity. 
In order to determine this parity,
we need to consider the complex~$\Delta_{n}$,
displayed in Figure~\ref{figure_one}, and then apply
Theorem~\ref{theorem_Euler}.

\begin{figure}[t]
\setlength{\unitlength}{1.5mm}
\begin{center}
\begin{picture}(20,18)(0,0)
\put(5,0){\circle*{1}}
\put(15,0){\circle*{1}}
\put(0,8.66){\circle*{1}}
\put(20,8.66){\circle*{1}}
\put(5,17.32){\circle*{1}}
\put(15,17.32){\circle*{1}}
\qbezier(5,0)(5,0)(15,0)
\qbezier(15,0)(15,0)(20,8.66)
\qbezier(20,8.66)(20,8.66)(15,17.32)
\qbezier(15,17.32)(15,17.32)(5,17.32)
\qbezier(5,17.32)(5,17.32)(0,8.66)
\qbezier(0,8.66)(0,8.66)(5,0)
\put(1,18.5){$2^{b}+2^{a}$}
\put(16,18.5){$2^{a}$}
\put(-3,8){$2^{b}$}
\put(1,-3){$2^{c}+2^{b}$}
\put(15,-3){$2^{c}$}
\put(21,8){$2^{c}+2^{a}$}
\end{picture}
\end{center}
\caption{The complex $\Delta_{n}$ for $n = 2^{c} + 2^{b} + 2^{a}$.
Note that the essential elements are $2^{a},2^{b},2^{b}+2^{a},2^{c},2^{c}+2^{a},2^{c}+2^{b}$
in base $2$, corresponding to the vertices.}
\label{figure_one}
\end{figure}
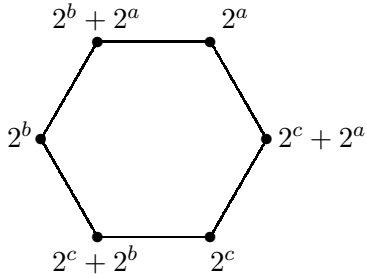

We now have four subcases to consider.
\begin{itemize}
\item[--]
First consider sets~$S$ such that
$S\cap \{2^{a},2^{b},2^{c}+2^{a},2^{c}+2^{b}\}=\{2^{a},2^{c}+2^{a}\}$.
Note that the four induced subcomplexes
$\Delta_{n}\vrule_{S}$,
$\Delta_{n}\vrule_{S \triangle \{2^{b}+2^{a}\}}$,
$\Delta_{n}\vrule_{S \triangle \{2^{c}\}}$ and
$\Delta_{n}\vrule_{S \triangle \{2^{b}+2^{a},2^{c}\}}$
are all contractible and hence have reduced Euler characteristic $0$.
Hence in this case
$\beta_{n}(S)$,
$\beta_{n}(S \triangle \{2^{b}+2^{a}\})$,
$\beta_{n}(S \triangle \{2^{c}\})$ and
$\beta_{n}(S \triangle \{2^{b}+2^{a},2^{c}\})$
are all even.

\item[--]
Second, when $S\cap \{2^{a},2^{b},2^{c}+2^{a},2^{c}+2^{b}\}=\{2^{b},2^{c}+2^{b}\}$,
by considering the reverse sets of the previous case,
the four sets
$S$,
$S \triangle \{2^{b}+2^{a}\}$,
$S \triangle \{2^{c}\}$ and
$S \triangle \{2^{b}+2^{a},2^{c}\}$
have even descent set statistics because their corresponding induced subcomplexes are contractible.

\item[--]
Third, consider sets~$S$ such that
$S\cap \{2^{a},2^{b},2^{c}+2^{a},2^{c}+2^{b}\}=\{2^{a},2^{c}+2^{b}\}$.
Now the four induced subcomplexes
$\Delta_{n}\vrule_{S}$,
$\Delta_{n}\vrule_{S \triangle \{2^{b}+2^{a}\}}$,
$\Delta_{n}\vrule_{S \triangle \{2^{c}\}}$ and
$\Delta_{n}\vrule_{S \triangle \{2^{b}+2^{a},2^{c}\}}$
are all homotopy equivalent to two points
and hence have reduced Euler characteristic $1$.
Hence in this case
the descent set statistics of the four sets
$S$,
$S \triangle \{2^{b}+2^{a}\}$,
$S \triangle \{2^{c}\}$ and
$S \triangle \{2^{b}+2^{a},2^{c}\}$
are all odd.

\item[--]
The fourth and last case is when
$S\cap \{2^{a},2^{b},2^{c}+2^{a},2^{c}+2^{b}\}=\{2^{b},2^{c}+2^{a}\}$.
Again, the four induced subcomplexes
$\Delta_{n}\vrule_{S}$,
$\Delta_{n}\vrule_{S \triangle \{2^{b}+2^{a}\}}$,
$\Delta_{n}\vrule_{S \triangle \{2^{c}\}}$ and
$\Delta_{n}\vrule_{S \triangle \{2^{b}+2^{a},2^{c}\}}$
are all homotopy equivalent to two points
and hence have reduced Euler characteristic $1$.
Therefore,
the descent set statistics of the four sets
$S$,
$S \cup \{2^{b}+2^{a}\}$,
$S \cup \{2^{c}\}$ and
$S \cup \{2^{b}+2^{a},2^{c}\}$
are all odd.
\end{itemize}
From these four subcases above we know that
$\beta_{n}(S) \equiv 1 + \beta_{n}(S \triangle \{2^{a},2^{b}\}) \bmod 2$.
Next, since $2^{a}$ and~$2^{b}$
both satisfy
$\binom{n}{2^{a}} \equiv \binom{n}{2^{b}} \equiv 0 \bmod s$,
we have that
$\beta_{n}(S) \equiv 
-\beta_{n}(S \triangle \{2^{a}\}) \equiv 
\beta_{n}(S \triangle \{2^{a},2^{b}\})  \bmod s$.
Combining these two statements and using that
$2s$ divides $\binom{n}{k}$ we conclude that
$$  \beta_{n}(S)
  \equiv s + \beta_{n}(S \triangle \{2^{a},2^{b}\}) 
  \equiv s - \beta_{n}(S \triangle \{2^{a},2^{b},k\})   \bmod 2s  . $$
Thus, the third and final case of the definition of $\phi$
is $\phi(S) = S \triangle \{2^{a},2^{b},k\}$.
This proves that the equality $a_{j} = a_{s-j}$.
With the proper equalities holding,
Lemma~\ref{factor_lemma} proves the theorem.
\end{proof}

One might ask if it is possible for $\Phi_{2s}$ to be a factor of $Q_{n}(t)$
if the binary expansion of $n$
has more than three binary digits.
Although the equations within Lemma~\ref{factor_lemma} are only sufficient
conditions and not necessary conditions for this cyclotomic polynomial
to be a factor, it is easy to see why
this lemma cannot be used unless there are equal numbers
of even and odd descent set statistics, as this
is implied by the combination of
equations~\eqref{equation_vertical}
and~\eqref{equation_horizontal}.
Chebikin et al.\ showed that there are not equal numbers
when the binary expansion of $n$ has 4 or 5 digits, although
it is not known if there is a $k>3$ for which this condition
is true when $n$ has $k$ binary digits.

Similar to Remark~\ref{remark_18} and~\ref{remark_20}
we have the next remark about $21$ and $22$.
\begin{remark}
{\rm
Consider $n = 21=2^{4} + 2^{2} + 1$ and $k=2$.
Observe that
$\gcd\left(\binom{21}{16}, \binom{21}{4}, \binom{21}{1}\right) = 21$.
Furthermore, observe that $\binom{21}{2}$ is a multiple of $21$.
Hence we obtain for each divisor $s$ of $21$ that
the cyclotomic polynomial $\Phi_{2s}$ divides $Q_{21}(t)$.
Similarly, for $n=22=2^{4} + 2^{2} + 2^{1}$ and $k=3$,
we have that
$\gcd\left(\binom{22}{16}, \binom{22}{4}, \binom{22}{2}\right) = 77$
divides $\binom{22}{3}$. Hence 
for each divisor $s$ of $77$ we conclude that
$\Phi_{2s}$ divides~$Q_{22}(t)$.}
\label{remark_21_22}
\end{remark}

We continue to consider the case when the integer $s$
is an odd prime $p$.
The following theorems give conditions for $p$
and the exponents $a$, $b$ and $c$ that provide the 
assumptions made for applying Theorem~\ref{theorem_three_digits}.

\begin{theorem}
\label{theorem_2_case_one}
Let $n=2^{c}+2^{b}+2^{a}$ where we assume $c>b>a$,
$n \geq 11$, and that the order~$g$ of~$2$ in
the multiplicative group $\Zzz_{p}^{*}$ is even.
If we have $\{a,b,c\} \equiv \{1,{g}/{2},{g}/{2}\} \bmod g$,
then~$2^{c}$, $2^{b}$ and $2^{a}$ are non-essential in base $p$.
Furthermore, the element $7$ is non-essential for
both base $2$ and base $p$.
Hence~$\Phi_{2p}$ is a factor of~$Q_{n}(t)$.
\end{theorem}
\begin{proof}
Using the same congruences as in the proof of
Theorem~\ref{theorem_even_order_one}, we have
\begin{align*}
n
= 2^{c}+2^{b}+2^{a}
\equiv 2^{g/2}+2^{g/2}+2^1 
\equiv (p-1)+(p-1)+2 
&\equiv 0 \bmod p, 
\end{align*}
hence the last digit in the base $p$ expansion of $n$ is $0$.
This makes $2^{c}$, $2^{b}$ and $2^{a}$ be non-essential in base
$p$ since the last digit for these powers of two are each greater
than 0.

The assumption that $n \geq 11$ implies that $c\geq 3$,
hence the number $7$ is non-essential
for $n$ in base~$2$.
Additionally, the last digit of the base $p$ expansion of $7$
is non-zero except when $p=7$,
but this case is not included for this theorem since the order of $2$
in $\Zzz_{7}^{*}$ is odd.
Therefore, $7$ is non-essential in base~$p$ as well,
which concludes the proof of the theorem.
\end{proof}

\begin{theorem}
\label{theorem_2_case_two}
Let $n=2^{c}+2^{b}+2^{a}$ where $c>b>a$, and assume that
$p$ is an odd prime greater than or equal to~$5$.
If $\{a,b,c\}\equiv \{g-2,g-2,g-1\} \bmod g$
where $g$ is the multiplicative order of~$2$ in~$\Zzz_{p}^{*}$,
then $2^{c}$, $2^{b}$ and $2^{a}$ are non-essential in base $p$.
Furthermore, the element $3$
is non-essential in both base $2$ and base $p$.
Hence~$\Phi_{2p}$ is a factor of~$Q_{n}(t)$.
\end{theorem}
\begin{proof}
We have 
\begin{align*}
n
= 2^{c}+2^{b}+2^{a}
\equiv 2^{g-2}+2^{g-2}+2^{g-1}
&\equiv 2^{g} \equiv 1 \bmod p,
\end{align*}
hence the last digit of the base $p$ expansion of $n$ is 1.  
Since we assume $p \geq 5$, we must have $g \geq 3$, hence
$2^{g-1}>2^{g-2}>1 \bmod p$. Thus, $2^{c}$, $2^{b}$ and $2^{a}$ are non-essential for~$n$ in base $p$ since the last digit of their base $p$ expansions 
larger than 1.

Also since we assume $p \geq 5$,
the last digit of the base $p$ expansion of $3$  is greater than $1$
as well, making it non-essential in base $p$.  The element $3$ is also non-essential in base $2$ since the fact that $g \geq 3$ implies that $2^{a}\neq 1$.
\end{proof}

\begin{remark}
\label{remark_four}
{\rm
With $n$ having two binary digits, there were many equivalencies modulo $m$ on the exponents $a$ and $b$
beyond what could be shown in results that held for all $p$ or for all $p$ with $m$ being even.  Likewise,
many such equivalencies exist in the three binary digit case that cause 
each of $2^{c}$, $2^{b}$ and $2^{a}$ to be non-essential in base $p$.
Examples of these equivalencies include the following:
\begin{itemize}
\item[--]
$\{a,b,c\}\equiv \{0,0,0\} \bmod 2$ when $p=3$ since the final digit of
$n$ in base $3$ is $1+1+1\equiv 0 \bmod 3$

\item[--]
$\{a,b,c\} \equiv \{1,2,4\} \bmod 10$
when $p=11$ because the last digit of $n$ is $2+4+16\equiv 0 \bmod 11$

\item[--]
$\{a,b,c\}\equiv \{1,2,3\}\bmod 12$ 
when $p=13$ since the last digit of $n$ in base $13$ is
$2+4+8 \equiv 1 \bmod 13$.  
\end{itemize}
Of course, to show that $Q_{n}(t)$
has $\Phi_{2p}$ as a factor, one still needs to find
an element $k$ that is non-essential in base $2$ and $p$, which
is shown for these examples for
a particular $n$ value in Table~\ref{table_factors_3}.
}
\end{remark}

\begin{remark}
\label{remark_five}
{\rm
As with Remark~\ref{remark_three}, we can also find equivalencies
modulo $G$ for the exponents 
$a,$ $b$ and $c$ when $n$ has three binary digits.
As an example, when $p=3$ there are equivalencies 
such as  $\{a,b,c\}\equiv \{3,4,5\} \bmod 6$ that cause $2^{c}$, $2^{b}$ 
and $2^{a}$ to be non-essential in base $p$.
This one exists because the last two digits of $n$ are $22+21+12\equiv 02$,
whereas the second to last digit of $2^{c}$, $2^{b}$ 
and~$2^{a}$ is each greater than 0.
}
\end{remark}

\begin{theorem}
\label{theorem_2_case_three}
Let $n=2^{c}+2^{b}+2^{a}$ where $c>b>a$ and $n>7$, and assume that
$p=2^e+2^d+1$ where $e>d$.  If $\{a,b,c\}\equiv \{0,d,e\} \bmod g$ where
$g$ is the multiplicative order of $2$ in $\Zzz_{p}^{*}$, then 
$2^{c}$, $2^{b}$ and $2^{a}$ are non-essential
in base $p$.  Furthermore, at least one of the elements $7$ or $13$ is non-essential in both base~$2$ and base~$p$.
Hence~$\Phi_{2p}$ is a factor of~$Q_{n}(t)$.
\end{theorem}
\begin{proof}
We have
$$ n = 2^{c}+2^{b}+2^{a} \equiv 2^{e}+2^{d}+1 \equiv p \equiv 0 \bmod p,$$
making the last digit in the base $p$ expansion of $n$ be $0$.  This causes $2^{c}$, $2^{b}$ and $2^{a}$ to be non-essential
in base $p$ since their last digits are $1$, $2^d$ or $2^e$, 
all of which are greater than $0 \bmod p$. 

First consider when $p\neq 7$.  In this case, the element $7$ is
non-essential in base $p$ because the last
digit in its base $p$ expansion is greater than $0$,
which is the last digit for $n$.  Since we assume $n>7$ with 
three binary digits, the element $7$ is also non-essential in base $2$
since $7$ also three digits in its 
binary expansion, completing the result in this case.

If we instead assume $p=7$, then the element $13$ is non-essential
in base $7$ since its base $7$ expansion
has a $6$ as its final digit.
The assumption of $n>7$,
the fact that $d=1$ and $e=2$, and that $\{a,b,c\}\equiv \{0,d,e\} \bmod 3$ 
where the $3$ is the order of $2$ in $\Zzz_{7}^{*}$, result in
the smallest such value for $n$ being $14$.
Since $n$ and $13$ each have
three binary digits with $n>13$,
we have that $13$ is also non-essential in base $2$,
concluding the proof of the theorem.
\end{proof}

The following proposition explains the occurrence of another cyclotomic
factor of the form $\Phi_{2p}$ which is an outlier compared to other such
factors.  When $n=11$ and $p=3$, observe from the base~$2$
and base $3$ 
expansions of $11 = 2^{3} + 2 + 1 = 3^{2} + 2 \cdot 1$ that
both $2$ and $1$ are essential
in base $p$.  Therefore, Theorem~\ref{theorem_three_digits}
is not applicable.
However, $\Phi_{6}$ is still a factor of the descent set polynomial
for $n=11$, as shown in Proposition~\ref{proposition_eleven},
but we first need the 
following lemma to obtain certain descent set statistics modulo 3.

\begin{lemma}
Let $R$ be a subset of the interval $[3,8]$.
Then we have the following four evaluations of
descent set statistics:
\begin{align*}
\beta_{11}(R \cup \{1,9\})
& \equiv
\beta_{11}(R \cup \{2,10\})
\equiv
(-1)^{|R|} \bmod 3, \\
\beta_{11}(R \cup \{1,10\})
& \equiv
\beta_{11}(R \cup \{2,9\})
\equiv
- (-1)^{|R|} \bmod 3. 
\end{align*}
Especially, all these values are non-zero modulo $3$.
\label{lemma_eleven}
\end{lemma}
\begin{proof}
We consider
the quasi-symmetric function of $B_{11}$ modulo $3$.
Using Lemma~\ref{lemma_power_prime}, we have
\begin{align*}
F(B_{11})
&\equiv M_{(9)}\cdot M_{(1)}^2 \\
&\equiv M_{(9)}\cdot (M_{(2)}+2M_{(1,1)}) \\
&\equiv M_{(11)}+M_{(9,2)}+M_{(2,9)}+2M_{(9,1,1)}+2M_{(10,1)}\\
& \hspace{.6cm} +2M_{(1,9,1)}+2M_{(1,10)}+2M_{(1,1,9)} \bmod 3 ,
\end{align*}
where the second and third step is expanding a product
of monomial quasi-symmetric functions in terms of
monomial quasi-symmetric functions;
see~\cite[Lemma~3.3]{Ehrenborg_Hopf}.
Reading of the coefficients
of the quasi-symmetric functions, we have
the following values for the flag $f$-vector:
$$f_{S} \equiv 
\begin{cases}
1 & \text{if }S= \emptyset,\text{ } \{9\},\text{ or }\{2\}, \\
2 & \text{if }S=\{9,10\},\text{ }\{10\},\text{ }\{1,10\},
\text{ }\{1\},\text{ or }\{1,2\}, \\
0 & \text{otherwise.}
\end{cases}
\:\:\:\: \:\:\:\: \bmod 3 . 
$$
Observe that only eight entries are non-zero modulo $3$.
Using inclusion-exclusion, the descent set statistic
is given by
\begin{align*}
\beta_{11}(R \cup \{1,9\})
& \equiv
\sum_{T \subseteq R \cup \{1,9\}}
(-1)^{|R \cup \{1,9\} - T|} \cdot f_{T} \\
& \equiv
(-1)^{|R \cup \{1,9\}|} \cdot f_{\emptyset}
+
(-1)^{|R \cup \{9\}|} \cdot f_{\{1\}}
+
(-1)^{|R \cup \{1\}|} \cdot f_{\{9\}} \\
& \equiv
(-1)^{|R|}  \bmod 3 .
\end{align*}
The three descent set statistics
$\beta_{11}(R \cup \{1,10\})$,
$\beta_{11}(R \cup \{2,9\})$ and
$\beta_{11}(R \cup \{2,10\})$
can be computed similarly.
\end{proof}

\begin{proposition}
The cyclotomic polynomial $\Phi_{6}$ is a factor of
the descent set polynomial $Q_{11}(t)$.
\label{proposition_eleven}
\end{proposition}
\begin{proof}
Observe from the base $2$ and base $3$ expansions of 11 that 3 is non-essential for $11$ in base 2 and 
in base $3$.  Therefore, Lemma~\ref{lemma_flip_sign} implies that $a_j= a_{-j}$ for all $j$, or $a_1=a_5$ and $a_2=a_4$. We next focus on showing $a_0= a_3$ before proving $a_j=a_{3-j}$ for all other $j$.

Similarly to equations~\eqref{equation_a_new}
and~\eqref{equation_b_new}, since 
$8$ and 3 are essential for $11$ in base 2 but non-essential in base $3$,
we have 
\begin{align*}
\beta_{11}(S)+\beta_{11}(S \triangle \{3\})
&\equiv \begin{cases}
       0 & \text{if } |S\cap\{1,2\}|=1, \\
       3 & \text{if } |S\cap\{1,2\}|=0 \text{ or } 2
           \end{cases}
\:\:\:\:
          \bmod 6 , \\      
\beta_{11}(S)+\beta_{11}(S \triangle \{8\})
&\equiv
\begin{cases}
       0 & \text{if } |S\cap\{9,10\}|=1, \\
       3 & \text{if } |S\cap\{9,10\}|=0 \text{ or } 2
\end{cases}
\:\:\:\:
          \bmod 6.   
\end{align*}

Assume $S\subseteq [10]$ in which $\beta_{11}(S)\equiv 0 \bmod 3$.  As in 
Theorem~\ref{theorem_three_digits}, if $|S\cap \{1,2\}|=0$ or $2$, or if
$|S\cap \{9,10\}|=0$ or $2$, the descent set statistics $\beta_{11}(S)$,
$\beta_{11}(S\triangle \{3\})$, $\beta_{11}(S\triangle \{8\})$, and 
$\beta_{11}(S\triangle \{3,8\})$ contribute evenly between $a_0$ and $a_3$.

On the other hand, if $|S\cap \{1,2\}|=1$ and $|S\cap \{9,10\}|=1$, then 
$S$ is one of the four sets
in Lemma~\ref{lemma_eleven}.
Therefore, the descent set statistic of the set $S$ is non-zero modulo $3$,
and does not contribute
to either $a_{0}$ or $a_{3}$.
In conclusion, the only possible sets that do contribute
to $a_{0}$ and $a_{3}$ do so evenly, so $a_{0} = a_{3}$.

It remains to show $a_{1} = a_{2}$ and $a_{4} = a_{5}$.
Since $11$ has three digits in its binary expansion,
Corollary~\ref{corollary_binary_exp} gives that 
$a_0+a_2+a_4=a_1+a_3+a_5$.
Combining this equality with 
$a_{0} = a_{3}$, $a_{1} = a_{5}$ and $a_{2} = a_{4}$,
it follows that $a_{1}=a_{2}$ and $a_{4}=a_{5}$. 
Thus, Lemma~\ref{factor_lemma} 
implies that $\Phi_{6}$ is a factor of~$Q_{11}(t)$.
\end{proof}

This result is particular to $n=11$.
Attempts to generalize to values of $n$ of
the form $2^{c}+2+1=p^{r}+2$ have so far failed.
For these $n$ one can similarly show that $a_{0}=a_{p}$.
Unfortunately, this does not imply $a_{j}=a_{p-j}$,
which is in fact not true for all $j$.

\begin{table}
\begin{center}
  \begin{tabular}{ | c | r | c | c | l | l |}
  \hline
$n$ & $s$ & $\{a,b,c\} \bmod g$ & $k$ & Chebikin et & Our \\ 
&&&& al.\ statement & statement \\ \hline
    \hline
    11 & 3 &  --- & 3
    && Prop.~\ref{proposition_eleven}
        \\ \hline
    11 & 11 & $\{0,1,3\}$ & 7 & 
    Thm.~5.5 & Thm.~\ref{theorem_2_case_three}  \\ \hline
    13 & 13 & $\{0,2,3\}$ & 7 & 
    Thm.~5.5 & Thm.~\ref{theorem_2_case_three}  \\ \hline
    14 & 7 & $\{0,1,2\}$ & 13  & 
    Thm.~5.6 & Thm.~\ref{theorem_2_case_three} \\ \hline
    14 & 13 & $\{1,2,3\}$ & 3 
    &&  Rem.~\ref{remark_four} \\ \hline
    14 & 91 && 3 
    &&  Thm.~\ref{theorem_three_digits} \\ \hline
    19 & 19 & $\{0,1,4\}$ & 7 &
    Thm.~5.5 & Thm.~\ref{theorem_2_case_three} \\ \hline
    21 & 3 & $\{0,0,0\}$ & 2 
    && Rem.~\ref{remark_four} \\ \hline
    21 & 7 & $\{0,1,2\}$ & 13 
    && Thm.~\ref{theorem_2_case_three}  \\ \hline
    21 & 21 && 2 
    && Rem.~\ref{remark_21_22}  \\ \hline
    22 & 7 & $\{1,1,2\}$ & 3 
    && Thm.~\ref{theorem_2_case_two} \\ \hline
    22 & 11 & $\{1,2,4\}$ & 7 &
    Thm.~5.6 & Rem.~\ref{remark_four}  \\ \hline
    22 & 77 && 3
    && Rem.~\ref{remark_21_22}  \\ \hline \hline    
    56 & 3 & $\{3,4,5\} \bmod 6$ & 3
    && Rem.~\ref{remark_five} \\ \hline
    4,108 & 13 & $\{0,2,3\}$ & 7 
    && Thm.~\ref{theorem_2_case_three} \\ \hline
    16,576 & 17 & $\{6,6,7\}$ & 3 
    && Thm.~\ref{theorem_2_case_two} \\ \hline
    32,802& 11 & $\{1,5,5\}$ & 7 
    && Thm.~\ref{theorem_2_case_one} \\ \hline
\end{tabular}
\end{center}
\caption{Examples of cyclotomic factors of 
$Q_{n}(t)$ of the form $\Phi_{2s}$
where the binary expansion of $n$ has three $1$'s.}
\label{table_factors_3}
\end{table}

Table~\ref{table_factors_3} summarizes particular
values of $n$ and $s$ for which $\Phi_{2s}$ is a factor of 
$Q_{n}(t)$ with $n$ having three binary digits, as was
done in Table~\ref{table_factors_2} for $n$ with two binary
digits.

\section{The double factor $\Phi_{2}$ in the descent set polynomial}
\label{section_double_factor_Phi_2}

Our next results are about
the occurrence of the double factors in
the descent set polynomial $Q_{n}(t)$.
Here we sharpen techniques of Chebikin et al.\
to explain more double factors.

We begin by recalling the $\ab$- and the $\cd$-index of the
Boolean algebra.
Let $\zab$ denote the polynomial ring in the non-commutative
variables $\av$ and $\bv$. For $S$ a subset of $[n-1]$
define the $\ab$-monomial $u_{S} = u_{1} u_{2} \cdots u_{n-1}$
where $u_{i} = \av$ if $i \not\in S$
and $u_{i} = \bv$ if $i \in S$.
The polynomial $\Psi(B_{n})$ given by
$$  \Psi(B_{n})
  =
    \sum_{S \subseteq [n-1]} \beta_{n}(S) \cdot u_{S}  ,  $$
is known as the $\ab$-index of the Boolean algebra.
The result we need is that $\Psi(B_{n})$ can be written in terms
of $\cv = \av + \bv$ and $\dv = \av\bv + \bv\av$, which
is originally due to
Bayer and Klapper~\cite{Bayer_Klapper}.
For ways to compute $\Psi(B_{n})$
see~\cite[Proposition~8.2]{Billera_Ehrenborg_Readdy}
and~\cite{Ehrenborg_Readdy_coproducts}.
For more details 
see also Theorem~1.6.3
or Section~3.17 in~\cite{EC1}.

Define a linear function $\mathcal{L}$ from $\zab$
to $\Zzz$ by 
\begin{align*}
  \mathcal{L}(u_{S}) & = (-1)^{\beta_{n}(S)}  , 
\end{align*}
where $S$ is a subset of $[n-1]$ and $u_{S}$ is
the associated $\ab$-monomial of degree $n-1$.
Note that we abuse notation such that
for an $\ab$-monomial $u$ of degree $n-1$,
we write
$\beta_{n}(u)$ instead of $\beta_{n}(S)$,
where $u = u_{S}$.

\begin{proposition}
Let $w$ be a $\cd$-monomial of degree $2n - 1$
having $j$ $\dv$'s.
Then the following evaluation holds
$$   \mathcal{L}(w)
    =
       2^{2n-j-1} \cdot (1 - 2 \cdot \rho(n)) . $$
\label{proposition_2n_n}
\end{proposition}
\begin{proof}
Let $u = u_{1} u_{2} \cdots u_{2n-1}$ be
an $\ab$-monomial in the expansion of $w$.
Let $v$ be the $\ab$-monomial formed by taking the letters
in even positions from $u$,
that is, $v = u_{2} u_{4} \cdots u_{2n-2}$.
By Theorem~\ref{theorem_Euler} we have that
$$    \beta_{2n}(u)
     \equiv
         \widetilde{\chi}(\Delta_{2n}\vrule_{S})
     \equiv
         \widetilde{\chi}(\Delta_{n}\vrule_{T})
     \equiv
         \beta_{n}(v) \bmod 2  ,  $$
since the two complexes
$\Delta_{2n}\vrule_{S}$
and
$\Delta_{n}\vrule_{T}$
are identical 
where $u = u_{S}$ and $v = u_{T}$.
Furthermore,
observe that every $\ab$-monomial of degree $n-1$
appears this way.

Given an $\ab$-monomial $v$
of degree $n-1$, how many corresponding monomials $u$
can we find within the expansion of
the $\cd$-monomial $w$?
There are $n$ odd positions in $u$ to fill in.
If an odd position is covered by a $\dv$ in $w$,
there is a unique way to fill it in.
Note that there are $n-j$ odd positions in $u$ associated with
$\cv$'s in $w$.
Hence there are $2^{n-j}$ ways to fill in $v$ to get an
$\ab$-monomial $u$ in the expansion of $w$.
Now
\begin{align*}
\mathcal{L}(w)
& =
\sum_{u} (-1)^{\beta_{2n}(u)} \\
& =
2^{n-j} \cdot \sum_{v} (-1)^{\beta_{n}(v)} \\
& =
2^{n-j} \cdot Q_{n}(-1) \\
& =
2^{2n-j-1} \cdot (1 - 2 \cdot \rho(n)) ,
\end{align*}
where the first sum is over all $\ab$-monomials $u$
occurring in the expansion of $w$
and the second sum is over all $\ab$-monomials $v$ of degree $n-1$.
\end{proof}

\begin{theorem}
If $\Phi_{2}$ is a factor of $Q_{2n}(t)$
then $\Phi_{2}$ is a double factor of $Q_{2n}(t)$.
\label{theorem_double_factor}
\end{theorem}
\begin{proof}
Observe that
\begin{align*}
Q^{\prime}_{2n}(t)
& =
\sum_{S} \beta_{2n}(S) \cdot t^{\beta_{2n}(S) - 1} .
\end{align*}
Hence evaluating $Q^{\prime}_{2n}(t)$ at $t = -1$, we obtain
\begin{align*}
Q^{\prime}_{2n}(-1)
& =
- \sum_{S} \beta_{2n}(S) \cdot (-1)^{\beta_{2n}(S)} \\
& =
- \mathcal{L}\left( \sum_{S} \beta_{2n}(S) \cdot u_{S} \right) \\
& =
- \mathcal{L}( \Psi(B_{2n}) ) .
\end{align*}
Now if $\Phi_{2}$ is a factor of $Q_{2n}(t)$, we have $\rho(n) = 1/2$.
Since $\Psi(B_{2n})$ can be expressed in terms of the two
variables $\cv$ and $\dv$, we conclude that
$\mathcal{L}(\Psi(B_{2n})) = 0$.
Hence $-1$ is a double root of $Q_{2n}(t)$,
yielding the conclusion.
\end{proof}

Now extending Theorem~7.3 in~\cite{C_E_P_R}
we have the next result.
\begin{corollary}
If the binary expansion of $n$ has three $1'$s
then $\Phi_{2}^{2}$ divides $Q_{2n}(t)$.
\end{corollary}

\section{The double factor $\Phi_{2p}$ in $Q_{2q}(t)$}

Throughout this section,
assume $q$ is an odd prime power, that is, $q=p^r$ where $p$ is prime and
$r$ is a positive integer.

Observe that by Theorem~6.1, part (iv) in~\cite{C_E_P_R}
the cyclotomic polynomial $\Phi_{2p}$
is a factor of the descent set polynomial $Q_{2q}(t)$.
Hence we concentrate on extending Theorem~7.5 from~\cite{C_E_P_R}
to show that $\Phi_{2p}$ is a double factor in this section.

\begin{theorem}
If $\rho(q) = 1/2$,
then the cyclotomic polynomial $\Phi_{2p}$
is a double factor of the descent set polynomial $Q_{2q}(t)$.
\label{theorem_double_factor_two}
\end{theorem}

In order to prove this theorem we introduce two new
linear functions
$\mathcal{C}$ and $\mathcal{S}$
from $\ab$-polynomials of degree $2q-1$ to the real field $\Rrr$ by
\begin{align}
\mathcal{C}(u_{S}) & = \cos(\pi/p \cdot \beta_{2q}(S)) , 
\label{equation_def_C_1} \\
\mathcal{S}(u_{S}) & = \sin(\pi/p \cdot \beta_{2q}(S)) .
\label{equation_def_S_1}
\end{align}
Our goal is to show that $\mathcal{C}(w) = \mathcal{S}(w) = 0$
for any $\cd$-monomial $w$
of degree $2q-1$.
We do this by a series of lemmas.
First from Corollary~5.3
in~\cite{C_E_P_R}, we have the following result.
\begin{lemma}
The descent set statistic~$\beta_{2q}$ modulo
$p$ is given by
$$   \beta_{2q}(S) \equiv (-1)^{|S-\{q\}|} \bmod p.$$
\label{lemma_q}
\end{lemma}
\noindent This is straightforward to show using that
$(x_{1}+x_{2}+\cdots)^{2q} \equiv
  (x_{1}^{q}+x_{2}^{q}+\cdots)^{2} \equiv
  M_{(2q)} + 2M_{(q,q)} \bmod p$.

\begin{lemma}
For any $\ab$-monomial $u$ of degree $2q-1$,
we have $\mathcal{C}(u) = - \cos(\pi/p) \cdot (-1)^{\beta_{2q}(u)}$.
\end{lemma}
\begin{proof}
According to Lemma~\ref{lemma_q} there are only four possible values for
$\beta_{2q}(u) \bmod 2p$.
When $\beta_{2q}(u)$ is odd, then the only two values
for $\beta_{2q}(u)$ modulo $2p$ are $\pm 1$,
in which case $\mathcal{C}(u)$ is $\cos(\pi/p)$.
Similarly, when $\beta_{2q}(u)$ is even,
it can only take the values $p \pm 1$,
and hence $\mathcal{C}(u)$ is $-\cos(\pi/p)$.
\end{proof}

\begin{lemma}
If $\rho(q) = 1/2$, then for a $\cd$-monomial $w$,
we have $\mathcal{C}(w) = 0$.
\label{lemma_C}
\end{lemma}
\begin{proof}
Assume that the $\cd$-monomial $w$ has $j$ $\dv$'s.
Now by the previous lemma we have
\begin{align*}
\mathcal{C}(w)
& =
\sum_{u} \mathcal{C}(u) \\
& =
-\cos(\pi/p) \cdot \sum_{u} (-1)^{\beta_{2q}(u)} \\
& =
-\cos(\pi/p) \cdot \mathcal{L}(w) \\
& =
-\cos(\pi/p) \cdot 2^{2q-j-1} \cdot (1 - 2 \cdot \rho(q)) .
\end{align*}
Since $\rho(q) = 1/2$, we obtain the conclusion $\mathcal{C}(w) = 0$.
\end{proof}

\begin{lemma}
Let $u$ and $v$ be two $\ab$-monomials such that
$\deg(u) + \deg(v) = 2q-2$, both
$\deg(u)$ and $\deg(v)$ are even,
and both $\deg(u)$ and $\deg(v)$ differ from $q-1$.
Then the functional $\mathcal{S}$ applied to
$u \cdot \cv \cdot v$ is zero, that is, $\mathcal{S}(u \cdot \cv \cdot v) = 0$.
\end{lemma}
\begin{proof}
Since $\deg(u) + 1$ is non-essential for $2 \cdot q$ both
in base $2$ and in base $p$,
we have by Lemma~\ref{lemma_flip} that
$$  \beta_{2q}(u \cdot \av \cdot v) \equiv
     -\beta_{2q}(u \cdot \bv \cdot v) \bmod 2p  .  $$
Since $\sin$ is an odd function, this identity directly implies
$\mathcal{S}(u \cdot \av \cdot v) = -\mathcal{S}(u \cdot \bv \cdot v)$.
\end{proof}

\begin{lemma}
Let $w$ be a $\cd$-monomial of degree $2q-1$
different from the monomial
$\dv^{(q-1)/2} \cv \dv^{(q-1)/2}$.
Then $\mathcal{S}(w)$ vanishes.
\end{lemma}
\begin{proof}
The monomial $w$ has $q$ odd positions
and $q-1$ even positions.
Since a $\dv$ covers both an odd position and an even
position, there will always be a $\cv$ in an odd position.
Unless $w$ is the monomial
$\dv^{(q-1)/2} \cv \dv^{(q-1)/2}$
we can find a $\cv$ in an odd position different from $q$.
By the previous lemma we know
$\mathcal{S}(u \cdot \cv \cdot v) = 0$
for all $\ab$-monomials $u$ and $v$ and
hence by linearity we conclude $\mathcal{S}(w) = 0$.
\end{proof}

\begin{lemma}
If $\rho(q) = 1/2$ then $\mathcal{S}(\dv^{(q-1)/2} \cv \dv^{(q-1)/2}) = 0$.
\end{lemma}
\begin{proof}
If $u$ is an $\ab$-monomial occurring in the expansion of
$w = \dv^{(q-1)/2} \cv \dv^{(q-1)/2}$
then it has $q-1$ or $q$ $\bv$'s.
In fact, it has $q-1$ $\bv$'s
in the positions different from the position $q$
since this is the position of the $\cv$ in $w$.

Lemma~\ref{lemma_q} implies that
$\beta_{2q}(u) \equiv (-1)^{q-1} \equiv 1 \bmod p$,
also using that $q$ is odd.
Hence modulo~$2p$ we have that
$\beta_{2q}(u) \equiv 1$ or $p+1 \bmod 2p$.
That is, the value of $\beta_{2q}(u)$ modulo $2p$ only depends on the
value modulo $2$. 
Hence we have the sum
\begin{align*}
\mathcal{S}(w)
& =
\sum_{u} \mathcal{S}(u) \\
& =
\sum_{u} \sin(\pi/p \cdot \beta_{2q}(u)) \\
& =
\sum_{u} -\sin(\pi/p) \cdot (-1)^{\beta_{2q}(u)} \\
& =
-\sin(\pi/p) \cdot \mathcal{L}(w) \\
& =
- \sin(\pi/p) \cdot
2^{q} \cdot (1 - 2 \cdot \rho(q))  .
\end{align*}
Since $\rho(q) = 1/2$, we obtain $\mathcal{S}(w) = 0$.
\end{proof}

\begin{proof}[Proof of Theorem~\ref{theorem_double_factor_two}]
Observe that
\begin{align*}
e^{\pi \cdot i/p} \cdot Q_{2q}^{\prime}(e^{\pi \cdot i/p})
& = 
\sum_{S} \beta_{2q}(S) \cdot e^{\beta_{2q}(S) \cdot \pi \cdot i/p} \\
& = 
\sum_{S}
   \beta_{2q}(S) \cdot (\mathcal{C}(u_{S}) + i \cdot \mathcal{S}(u_{S})) \\
& = 
(\mathcal{C} + i \cdot \mathcal{S})
\left(
\sum_{S}
   \beta_{2q}(S) \cdot u_{S}
\right) \\
& = 
(\mathcal{C} + i \cdot \mathcal{S})
(\Psi(B_{2q})),
\end{align*}
which vanishes.
Hence $e^{\pi \cdot i/p}$
is a root of $Q_{2q}^{\prime}$,
so $e^{\pi \cdot i/p}$
is a double root of $Q_{2q}$.
\end{proof}

\section{The (double) factor $\Phi_{2p}$ in $Q_{q+1}(t)$}
\label{section_q_plus_1}

Let $q = p^{r}$ be an odd prime power, that is,
$p$ is an odd prime and $r$ a positive integer.
Now we study the case of the cyclotomic factor
$\Phi_{2p}$ in $Q_{q+1}(t)$.
\begin{theorem}
If $\rho(q) = 1/2$ then
the cyclotomic polynomial $\Phi_{2p}$
divides the descent set polynomial~$Q_{q+1}(t)$.
Furthermore, if $q \equiv 3 \bmod 4$, then
$\Phi_{2p}$ is a double factor in 
$Q_{q+1}(t)$.
\label{theorem_double_factor_three}
\end{theorem}

We start by explicitly expressing the flag $f$-vector
of the Boolean algebra $B_{q+1}$ modulo $p$:
$$  F(B_{q+1})
   \equiv (M_{(1)})^{q} \cdot M_{(1)}
   \equiv M_{(q)} \cdot M_{(1)}
   \equiv M_{(q+1)} + M_{(q,1)} + M_{(1,q)} \bmod p . $$
Hence the flag $f$-vector $f(S) \equiv 1 \bmod p$ if
$S$ is equal to $\emptyset$, $\{1\}$ or $\{q\}$,
and zero otherwise.
By inclusion-exclusion we obtain that the descent
set statistic modulo $p$ is given by
\begin{equation}
\beta_{q+1}(S)
\equiv
\begin{cases}
(-1)^{S}   & \text{ if } |S \cap \{1,q\}| = 0 , \\
0             & \text{ if } |S \cap \{1,q\}| = 1 , \\
- (-1)^{S} & \text{ if } |S \cap \{1,q\}| = 2,   \\
\end{cases}
\:\:\:\: \bmod p .
\label{equation_q_plus_1}
\end{equation}
In terms of $\ab$-monomials,
this result can be stated as
$\beta_{q+1}(\av \cdot v \cdot \bv) \equiv
\beta_{q+1}(\bv \cdot v \cdot \av) \equiv 0 \bmod p$
and 
$\beta_{q+1}(\av \cdot v \cdot \av) \equiv 
 -\beta_{q+1}(\bv \cdot v \cdot \bv) \equiv (-1)^{j} \bmod p$
where $v$ is an $\ab$-monomial of degree $q-2$
having $j$~$\bv$'s.

Similarly to the previous section, we use 
two linear functions from $\ab$-polynomials
of degree $q$ to the reals~$\Rrr$, defined by
\begin{align}
\mathcal{C}(u_{S}) & = \cos(\pi/p \cdot\beta_{q+1}(S)) , 
\label{equation_def_C_2} \\
\mathcal{S}(u_{S}) & = \sin(\pi/p \cdot\beta_{q+1}(S)) .
\label{equation_def_S_2} 
\end{align}
Note that they differ from definitions~\eqref{equation_def_C_1}
and~\eqref{equation_def_S_1} by replacing
the descent set statistic $\beta_{2q}$ by $\beta_{q+1}$.

\begin{lemma}
Let $w$ be a $\cd$-monomial of degree $q$
beginning or ending with the letter $\cv$.
If $\rho(q+1)=1/2$ then $\mathcal{C}(w) = 0$.
\label{lemma_q_C_c}
\end{lemma}
\begin{proof}
It is enough to consider the case when
$w$ begins with a $\cv$.
Let $u=u_{1} u_{2} \cdots u_{q}$
be an $\ab$-monomial in the expansion of~$w$.
If $u_{1}$ differs from $u_{q}$,
we have 
by~\eqref{equation_q_plus_1} that
$\beta_{q+1}(u) \equiv 0 \bmod p$.
Hence
$\mathcal{C}(u)
= \cos(\pi/p \cdot \beta_{q+1}(u))
= (-1)^{\beta_{q+1}(u)}$.
In the case in which the first and last letter of $u$ are the same,
we have that
$\beta_{q+1}(u) \equiv \pm 1 \bmod p$
by~\eqref{equation_q_plus_1}.
Hence $\beta_{q+1}(u)$ takes one of the four values
$\pm 1, p \pm 1$ modulo $2p$
and so
$\mathcal{C}(u) = \cos(\pi/p \cdot \beta_{q+1}(u))$
takes one of the two values
$\pm \cos(\pi/p)$.
Note that if $\beta_{q+1}(u)$ is even,
then $\beta_{q+1}(u)$ is $p \pm 1$ modulo $2p$
and hence $\mathcal{C}(u)$ is $-\cos(\pi/p)$.
Similarly, if $\beta_{q+1}(u)$ is odd we have
$\mathcal{C}(u)$ is $\cos(\pi/p)$.
To summarize these two cases when $u_{1} = u_{q}$, we have that
$\mathcal{C}(u)
= - \cos(\pi/p) \cdot (-1)^{\beta_{q+1}(u)}$.

Then we have the sum
\begin{align*}
\mathcal{C}(w)
& =
\sum_{u \: : \: u_{1} \neq u_{q}}
(-1)^{\beta_{q+1}(u)}
-
\cos(\pi/p) \cdot 
\sum_{u \: : \: u_{1} = u_{q}}
(-1)^{\beta_{q+1}(u)} , 
\end{align*}
where both sums are over all $\ab$-monomials $u$
in the expansion of $w$.
Let overline denote the involution
defined by 
$\overline{\av} = \bv$ and $\overline{\bv} = \av$.
In each of the sums, also include the term
$\overline{u_{1}} u_{2} \cdots u_{q}$.
Since $1$ is non-essential for $q+1$ in base $2$,
we have
$\beta_{q+1}(\overline{u_{1}} u_{2} \cdots u_{q}) \equiv
\beta_{q+1}(u) \bmod 2$.  Hence both sums will double
to give us
\begin{align*}
\mathcal{C}(w)
& =
\frac{1}{2} \cdot \sum_{u}
(-1)^{\beta_{q+1}(u)}
-
\cos(\pi/p) \cdot 
\frac{1}{2} \cdot \sum_{u}
(-1)^{\beta_{q+1}(u)} \\
& =
\frac{1}{2} \cdot
(1 - \cos(\pi/p)) \cdot
\mathcal{L}(w) , 
\end{align*}
where both sums are over all $u$ occurring in the expansion
of $w$.
This works since $w$ begins with the letter $\cv$.
By the assumption $\rho(q+1) = 1/2$,
this expression will vanish by
Proposition~\ref{proposition_2n_n}.
\end{proof}

\begin{lemma}
Let $w$ be a $\cd$-monomial of degree $q$
beginning or ending with the letter $\dv$.
If $\rho(q+1) = 1/2$ and $q \equiv 3 \bmod 4$
then $\mathcal{C}(w) = 0$.
\label{lemma_q_C_d}
\end{lemma}
\begin{proof}
Assume that $w$ begins with a $\dv$.
The proof is the same as the proof of the previous lemma,
except that $q \equiv 3 \bmod 4$
implies that $2$ is non-essential for $q+1$ in base $2$.
In the end of the proof when we extend
the two sums ranging over
$u = u_{1} u_{2} u_{3} \cdots u_{q}$,
also include the terms
$\overline{u_{1}} \overline{u_{2}} u_{3} \cdots u_{q}$.
Then the both sums become $\mathcal{L}(w)$
and the result follows.
\end{proof}

\begin{lemma}
Let $u$ and $v$ be two $\ab$-monomials such that
$\deg(u) + \deg(v) = q-1$, both
$\deg(u)$ and $\deg(v)$ are even,
and both $\deg(u)$ and $\deg(v)$ differ from zero.
Then the functional $\mathcal{S}$ applied to
$u \cdot \cv \cdot v$ is zero, that is, $\mathcal{S}(u \cdot \cv \cdot v) = 0$.
\end{lemma}
\begin{proof}
Since $\deg(u) + 1$ is non-essential for $q+1$ both
in base $2$ and in base $p$,
we have by Lemma~\ref{lemma_flip} that
$$
\beta_{q+1}(u \cdot \av \cdot v)  \equiv
- \beta_{q+1}(u \cdot \bv \cdot v) \bmod 2p.$$
Since $\sin$ is an odd function, this identity directly implies
$\mathcal{S}(u \cdot \av \cdot v) = -\mathcal{S}(u \cdot \bv \cdot v)$.
\end{proof}

\begin{lemma}
Let $w$ be a $\cd$-monomial of degree $q$
different from the monomials
$\cv \dv^{(q-1)/2}$,
$\dv^{(q-1)/2} \cv$, and 
$\cv \dv^{i} \cv \dv^{j} \cv$
where $i+j = (q-3)/2$.
Then $\mathcal{S}(w)$ vanishes.
\label{lemma_q_S_i}
\end{lemma}
\begin{proof}
If the monomial $w$ has a $\cv$ in an odd position $i$, where
$2 \leq i \leq q-1$, then $\mathcal{S}(w)$ vanishes by
the previous lemma.

The monomial $w$ has $(q+1)/2$ odd positions
and $(q-1)/2$ even positions.
Since a $\dv$ covers both an odd position and an even
position, there will always be a $\cv$ in an odd position.
However, this position could be position $1$ or position $q$.
In that situation, if there is only one $\cv$ in $w$, 
then it is either the monomial $\cv \dv^{(q-1)/2}$ or $\dv^{(q-1)/2} \cv$.
If there are three $\cv$'s in $w$ then two of $\cv$'s must be
the first and last positions. That is $w$ is of the form
$\cv \dv^{i} \cv \dv^{j} \cv$. Note that the middle $\cv$ is
in an even position and the previous lemma does not help.
\end{proof}

\begin{lemma}
Let $w$ be a $\cd$-monomial
of degree $q$ beginning and ending with the letter $\cv$.
Then $\mathcal{S}(w)$ vanishes.
In particular, $\mathcal{S}(\cv \dv^{i} \cv \dv^{j} \cv)= 0$
for $i+j = (q-3)/2$.
\label{lemma_q_S_ii}
\end{lemma}
\begin{proof}
Let $u$ be an $\ab$-monomial occurring in the expansion of $w$.
Observe that if $u$ has the form
$\av \cdot v \cdot \bv$ or
$\bv \cdot v \cdot \av$
then
$\beta_{q+1}(u) \equiv 0 \bmod p$
by~\eqref{equation_q_plus_1}.
This implies that
$\mathcal{S}(u) = \sin(\pi/p \cdot \beta_{q+1}(u)) = 0$.
Hence we have only to consider $\ab$-monomials
in the expansion of $w$ that begin and end with the same
letter.
Again by~\eqref{equation_q_plus_1} observe that
$\beta_{q+1}(\av \cdot v \cdot \av) \equiv
-\beta_{q+1}(\bv \cdot v \cdot \bv) \bmod p$.
Since positions $1$ and $q$ are non-essential for $q+1$
in base $2$, we have
$\beta_{q+1}(\av \cdot v \cdot \av) \equiv
\beta_{q+1}(\bv \cdot v \cdot \bv) \bmod 2$.
Combining these two congruences to one statement
modulo $2p$ we have
$\beta_{q+1}(\av \cdot v \cdot \av) \equiv
-\beta_{q+1}(\bv \cdot v \cdot \bv) \bmod 2p$.
This implies that
$\mathcal{S}(\av \cdot v \cdot \av) = -\mathcal{S}(\bv \cdot v \cdot \bv)$
and the statement of the lemma.
\end{proof}

\begin{lemma}
If $\rho(q+1) = 1/2$ and $q \equiv 3 \bmod 4$
then
$\mathcal{S}(\dv^{(q-1)/2} \cv) =
\mathcal{S}(\cv \dv^{(q-1)/2}) = 0$.
\label{lemma_q_S_iii}
\end{lemma}
\begin{proof}
The congruence relation on $q$ implies that
$4$ divides $q+1$. Hence the element $2$ is a non-essential
element for $q+1$ in base $2$. We will use this fact,
together with the facts that $1$ and $q$ are non-essential elements.

By symmetry it is enough to prove the lemma
for $w = \dv^{(q-1)/2} \cv$.
Let $u$ be an $\ab$-monomial occurring in the expansion of $w$.
If $u$ begins and ends with different letters,
similar to the previous lemma, we have that $\mathcal{S}(u) = 0$.
Hence we have that
$u$ has the form $\av\bv \cdot v \cdot \av$ or
$\bv\av \cdot v \cdot \bv$.
Next we have that
$\beta_{q+1}(\av\bv \cdot v \cdot \av)
\equiv
-\beta_{q+1}(\bv\bv \cdot v \cdot \bv)
\equiv
\beta_{q+1}(\bv\av \cdot v \cdot \bv) \bmod p$
by~\eqref{equation_q_plus_1}.
Furthermore, since the three elements
$1$,  $2$ and $q$ are non-essential for $q+1$ in base $2$,
we have that
$\beta_{q+1}(\av\bv \cdot v \cdot \av)
\equiv
\beta_{q+1}(\bv\av \cdot v \cdot \bv) \bmod 2$.
That is, modulo $2p$ we have
$\beta_{q+1}(\av\bv \cdot v \cdot \av)
\equiv
\beta_{q+1}(\bv\av \cdot v \cdot \bv) \bmod 2$p.

Hence these two cases $\av\bv \cdot v \cdot \av$
and $\bv\av \cdot v \cdot \bv$
are the same, that is,
\begin{align*}
\mathcal{S}(w)
& =
2 \cdot
\sum_{\av\bv \cdot v \cdot \av} 
\sin(\pi/p \cdot \beta_{q+1}(\av\bv \cdot v \cdot \av)) .
\end{align*}

\noindent The monomial $u = \av\bv \cdot v \cdot \av$ has
$(q-1)/2$ $\bv$'s,
so $\beta_{q+1}(u) \equiv (-1)^{(q-1)/2} \bmod p$ 
by Lemma~\ref{lemma_q}.
By considering the four values $\pm 1, p \pm 1$
of $\beta_{q+1}(u)$ modulo $2p$
we have that
$$ \sin(\pi/p \cdot \beta_{q+1}(u))
   =
     - (-1)^{(q-1)/2} \cdot \sin(\pi/p) \cdot
     (-1)^{\beta_{q+1}(u)}   . $$
Hence $\mathcal{S}(w)$ is given by
\begin{align*}
\mathcal{S}(w)
& =
-2 \cdot (-1)^{(q-1)/2} \cdot \sin(\pi/p) \cdot
\sum_{\av\bv \cdot v \cdot \av} 
     (-1)^{\beta_{q+1}(u)}   . 
\end{align*}
Again since the elements $1$, $2$ and $q$ are non-essential
for $q+1$ in base $2$, we can switch the letters in these
places without changing the descent set statistic $\beta_{q+1}$ modulo 2.
Hence we have
\begin{align*}
\mathcal{S}(w)
& =
-\frac{1}{2} \cdot (-1)^{(q-1)/2} \cdot \sin(\pi/p) \cdot
\sum_{u}
     (-1)^{\beta_{q+1}(u)}   ,
\end{align*}
where the sum is over all $\ab$-monomials $u$ in the expansion of $w$.
Now by the assumption that $\rho(q+1) = 1/2$
and Proposition~\ref{proposition_2n_n}, this
last sum is zero.
\end{proof}

\begin{proof}[Proof of Theorem~\ref{theorem_double_factor_three}]
Observe that
\begin{align*}
Q_{q+1}(e^{\pi \cdot i/p})
& = 
\sum_{u} e^{\pi \cdot i/p \cdot \beta_{q+1}(u)} \\
& = 
\sum_{u} (\cos(\pi/p \cdot \beta_{q+1}(u)) +
                 i \cdot \sin(\pi/p \cdot \beta_{q+1}(u))) \\
& =
(\mathcal{C} + i \cdot \mathcal{S})(\cv^{q}) ,
\end{align*}
since the first two sums is over all $\ab$-monomials
of degree $q$, that is, all the $\ab$-monomials in
the expansion of $\cv^{q}$. Finally, the last expression
vanishes by Lemmas~\ref{lemma_q_C_c} and~\ref{lemma_q_S_i}.

With the added assumption $q \equiv 3 \bmod 4$,
Lemmas~\ref{lemma_q_C_c}
and~\ref{lemma_q_C_d},
imply that $\mathcal{C}$
applied to any $\cd$-polynomial of degree $q$ vanishes.
Similarly, with the assumption
Lemmas~\ref{lemma_q_S_i} through~\ref{lemma_q_S_iii}
imply that $\mathcal{S}$
applied to any $\cd$-polynomial of degree $q$ vanishes.
Now we have that
\begin{align*}
e^{\pi \cdot i/p} \cdot Q_{q+1}^{\prime}(e^{\pi \cdot i/p})
& = 
\sum_{u} \beta_{q+1}(u) \cdot e^{\pi \cdot i/p \cdot \beta_{q+1}(u)} 
= 
(\mathcal{C} + i \cdot \mathcal{S})(\Psi(B_{q+1})) = 0 ,
\end{align*}
since $\Psi(B_{q+1})$ can be written in terms of the variables $\cv$ and $\dv$.
Thus
$e^{\pi \cdot i/p}$ is a double root of $Q_{q+1}(t)$.
\end{proof}

\section{Concluding remarks}

\newcommand{\bl}[1]{\mathbf{{\textcolor{blue}{#1}}}}
\begin{table}[ht!]
$$
\begin{array}{r r l}
 n & \mbox{\rm degree} & \mbox{\rm cyclotomic factors of $Q_{n}(t)$} \\ \hline
 3 &            2 & \bl{\Phi_{2}} \\
 4 &            5 & \bl{\Phi_{4}^{2}} \\
 5 &           16 & \bl{\Phi_{2}^{2}} \cdot \bl{\Phi_{10}} \\
 6 &           61 & \bl{\Phi_{2}^{2}} \cdot \bl{\Phi_{6}^{2}} \cdot \bl{\Phi_{10}} \\
 7 &          272 & \bl{\Phi_{2}} \\
 8 &         1385 & \bl{\Phi_{4}^{2}} \cdot \bl{\Phi_{28}} \\
 9 &         7936 & \bl{\Phi_{2}^{2}} \cdot \bl{\Phi_{6}} \cdot \bl{\Phi_{18}} \\
10 &        50521 & \bl{\Phi_{2}^{2}} \cdot \bl{\Phi_{6}} \cdot \bl{\Phi_{10}^{2}} \cdot \bl{\Phi_{18}} \cdot \bl{\Phi_{30}} \\
11 &       353792 & \bl{\Phi_{2}} \cdot \bl{\Phi_{6}} \cdot \bl{\Phi_{22}} \\
12 &      2702765
  & \bl{\Phi_{2}^{2}} \cdot \bl{\Phi_{6}} \cdot \bl{\Phi_{10}}
                  \cdot \bl{\Phi_{18}} \cdot \bl{\Phi_{22}^{2}}
                  \cdot \bl{\Phi_{66}} \cdot \bl{\Phi_{110}}
                  \cdot \bl{\Phi_{198}} \\
13 &     22368256
  & \bl{\Phi_{2}} \cdot \bl{\Phi_{26}} \\
14 &    1.993 \cdot 10^{8}
  & \bl{\Phi_{2}^{2}} \cdot \Phi_{4} \cdot \bl{\Phi_{14}^{2}} 
                           \cdot \bl{\Phi_{26}} \cdot \Phi_{28}
                          \cdot \bl{\Phi_{182}} \\
15 &   1.904 \cdot 10^{9}
  & - \\
16 &  1.939 \cdot 10^{10}
  &  \bl{\Phi_{4}^{2}} \cdot \bl{\Phi_{12}} \cdot
                          \bl{\Phi_{20}} \cdot \bl{\Phi_{44}} \cdot
                          \bl{\Phi_{52}} \cdot \bl{\Phi_{60}} \cdot
                          \bl{\Phi_{156}} \cdot \bl{\Phi_{220}} \cdot
                          \bl{\Phi_{260}} \cdot \bl{\Phi_{572}} 
                                                  \cdot \bl{\Phi_{2860}}       \\
17 & 2.099 \cdot 10^{11}
  &  \bl{\Phi_{2}^{2}} \cdot \bl{\Phi_{34}}
\\
18 & 2.405 \cdot 10^{12}
  &
\bl{\Phi_{2}^2} \cdot
\bl{\Phi_{6}^2} \cdot
\bl{\Phi_{18}} \cdot
\bl{\Phi_{34}} \cdot
\bl{\Phi_{102}} \cdot
\bl{\Phi_{306}} \\
19 & 2.909 \cdot 10^{13}
  &
\bl{\Phi_{2}} \cdot
 \bl{\Phi_{38}} \\
20 & 3.704 \cdot 10^{14}
  &
\bl{\Phi_{2}^2} \cdot
\bl{\Phi_{6}} \cdot
\bl{\Phi_{10}} \cdot
\bl{\Phi_{30}} \cdot
\bl{\Phi_{34}} \cdot
\bl{\Phi_{38}^{2}} \cdot
\bl{\Phi_{102}} \cdot
\bl{\Phi_{114}} \cdot
\bl{\Phi_{170}} \\
& &
\cdot
\bl{\Phi_{190}} \cdot
\bl{\Phi_{510}} \cdot
\bl{\Phi_{570}} \cdot
\bl{\Phi_{646}} \cdot
\bl{\Phi_{1938}} \cdot
\bl{\Phi_{3230}} \cdot
\bl{\Phi_{9690}} \\
21 & 4.951 \cdot 10^{15}
  &
\bl{\Phi_{2}} \cdot
\bl{\Phi_{6}} \cdot
\bl{\Phi_{14}} \cdot
\bl{\Phi_{42}} \\
22 & 6.935 \cdot 10^{16}
  &
\bl{\Phi_{2}^{2}} \cdot
\bl{\Phi_{14}} \cdot
\bl{\Phi_{22}^{2}} \cdot
\bl{\Phi_{154}} \\
23 & 1.015 \cdot 10^{18}
  &
-
\end{array}
$$
\caption{Cyclotomic factors of $Q_{n}(t)$.
This table is from Chebikin et al.~\cite{C_E_P_R},
but the explained factors have been updated.
These factors occur in {\bf {\textcolor{blue}{boldface}}}.
Furthermore the factor $\Phi_{2860}$ in $Q_{16}(t)$
has been included, which was missing in the original
table.
Note that the two factors
$\Phi_{4}$ and $\Phi_{28}$
in $Q_{14}(t)$ are still unexplained.}
\label{table_P}
\end{table}

By considering Table~\ref{table_P}
one sees that there are two
unexplained cyclotomic factors in
this table.
They are $\Phi_{4}$ and $\Phi_{28}$, both
dividing $Q_{14}(t)$.
Here it is straightforward to see 
$a_{4,1} = a_{4,3}$, that is, $Q_{14}(i)$ is a real
number. But it remains to find an argument
demonstrating that $a_{0} = a_{2}$.
Since $4$ is a square, the Chinese Remainder Theorem
cannot come to our rescue.
Note that these factors seems to isolated $n=14$ and
does not occur among other $n$ with three $1$'s in their
binary expansion.

Further consideration of Table~\ref{table_P}
shows that all cyclotomic factors that appear
in table with multiplicity 
have now been explained.
Are there other square factors appearing beyond
$n=23$ that have not yet been explained?

Do Theorems~\ref{theorem_double_factor_two}
and~\ref{theorem_double_factor_three}
apply to infinitely many prime powers?
As mentioned in the introduction,
there are only
$6$ prime powers with with two 1's in their
binary expansion.
However, there seems to be
an infinite number of primes with
three 1's in their binary expansion;
see the sequence A081091
in
The On-Line Encyclopedia of Integer Sequences.
However, this seems to be a hard number theory problem.

Chebikin et al.\ calculated the proportion for the number of
odd entries in the descent set statistics~$\beta_{n}$
for $n=1, 3, 7, 15, 31$; see Table~\ref{table_rho},
that is, for any integer with at
most five 1's in its binary expansion. 
Could the topological view of Theorem~\ref{theorem_Euler}
help for calculating the next case $n = 63$?
From this topological viewpoint, 
is there a classification of simplicial complexes $\Delta$
such that exactly half of the induced subcomplexes
$\Delta\vrule_{S}$ have an odd Euler characteristic?

In Chebikin et al.\ they also consider the 
signed descent set polynomial, that is,
$$    Q^{\pm}_{n}(t)
    =
        \sum_{S \subseteq [n]} t^{\beta^{\pm}_{n}(S)} ,  $$
where $\beta^{\pm}_{n}(S)$
is the number of signed permutations in
$\SSSS^{\pm}_{n}$ with descent set $S$.
Can any of our techniques be extended to explain
cyclotomic factors in this polynomial?
There are plenty of such factors; see Table~3
in~\cite{C_E_P_R}.

\section*{Acknowledgments}

The authors were partially supported by
National Security Agency grant~H98230-13-1-0280.

\newcommand{\journal}[6]{{\sc #1,} #2, {\it #3} {\bf #4} (#5) #6.}
\newcommand{\book}[4]{{\sc #1,} #2, #3, #4.}

\bigskip

{\em R.\ Ehrenborg, N.\ B.\ Fox.
Department of Mathematics,
University of Kentucky,
Lexington, KY 40506-0027,}
{\tt jrge@ms.uky.edu},
{\tt norman.fox@uky.edu}

\end{document}